\documentclass[12pt]{amsart}
\usepackage{mystyle}

\usepackage[T1]{fontenc}  

\numberwithin{equation}{section}

\begin{document}

\renewcommand{\bf}{\bfseries}
\renewcommand{\sc}{\scshape}

\vspace{0.5in}

\title[Limit cubic laminations]
{Limit cubic laminations}

\author[A.~Blokh]{Alexander~Blokh}

\thanks{The first named author was partially
supported by NSF grant DMS--2349942}

\thanks{The second named author was partially
supported by NSF grant DMS--1807558}

\author[L.~Oversteegen]{Lex Oversteegen}

\author[V.~Timorin]{Vladlen~Timorin}

\thanks{The third named author was supported by
the HSE University Basic Research Program.
}

\dedicatory{Dedicated to the memory of Bob Devaney}



\date{July 29, 2026}

\begin{abstract}
Let $\si_3:\uc\to \uc$ be the tripling map of the unit circle.
For sequences $\{\lam_i\}$ of $\si_3$-invariant dendritic laminations we study limits $(\oc, \od)$ of their critical
portraits assuming that one such limit $\Po=(\oc_\circ, \od_\circ)$ is given.
If the endpoints of $\oc_\circ$ and $\od_\circ$ are non-periodic, then there is a unique lamination $\lam$ with finite
critical sets such that $\oc$ and $\od$ can be any couple of critical chords compatible with $\lam$.
As the extreme opposite case we consider $\Po=(\ol{0 \frac13}, \ol{0 \frac23})$ and describe the corresponding
countable closed family of possible critical portraits $(\oc, \od)$ and the distinct laminations corresponding to them.
These results can be useful for the construction of a model for the cubic connectedness locus.
\end{abstract}

\maketitle

\section{Introduction: question and answer}
\label{s:new-intro}
Polynomials of one complex variable $z$ play a special role in holomorphic 1D dynamics.
Special properties of polynomials originate from the canonical form they take for large values of $z$.
Namely, by a theorem of B\"ottcher, any complex polynomial $f$ of degree $d>1$ is holomorphically conjugate
 to $w\mapsto w^d$ on a neighborhood of infinity, which allows one to define ``polar'' coordinates $(\theta(z),\rho(z))$
 of a point $z\approx\infty$ so that $f$ takes $z$ to the point with coordinates $(d\theta(z),d\rho(z))$.
Here, $w=e^{\rho+2\pi i\theta}$ is the complex coordinate in which $f$ takes the form $w\mapsto w^d$.
The coordinates $\theta(z)\in\R/\Z$ and $\rho(z)\in\R_{\ge 0}$ are called the \emph{argument} and the \emph{potential} of $z$.
It is possible to extend $\rho$ as a positive harmonic function on the \emph{basin of infinity} $\Omega_f:=\{z\mid f^n(z)\to\infty\}$.
Define the \emph{filled Julia set $K_f$ of $f$} as the complement $\C\sm\Omega_f$.
The most interesting dynamics happens on $K_f$ (more precisely, on the boundary $J_f$ of $\Omega_f$ called
the \emph{Julia set of $f$}).
Whenever all critical points of $f$ lie in $K_f$ (equivalently, when $K_f$ is connected),
the function $\theta$ can also be extended to a continuous function on $\Omega_f$ with values in $\R/\Z$;
level curves of this function are called \emph{external rays of $f$} while level curves of $\rho$ are called \emph{equipotentials of $f$}.

\emph{Combinatorics} of $K_f$ is easy to describe if $K_f$ is not only connected but also locally connected.
In this case, for every $\ta\in\R/\Z$, the external ray $R_f(\ta)$ of argument $\ta$ \emph{lands} on some point $z\in K_f$,
that is, $\{z\}=\ol{R_f(\ta)}\sm R_f(\ta)$. %and $z$ is the only boundary point of $R_f(\ta)$ in $\C$.
Write $G'_f(w)$ for the set of all $\ta\in\R/\Z$ such that $R_f(\ta)$ lands on $w\in K_f$.
Also, identifying $\R/\Z$ with $\uc:=\{z\in\C\mid |z|=1\}$, represent $G'_f(w)$ as a subset of the circle, and
take the convex hull $G_f(w):=\ch(G'_f(w))$ of this set in the plane (in the sense of real affine geometry); call $G_f(w)$ the \emph{gap-leaf} of $w$.
Simple topological considerations yield that $G_f(w_1)\cap G_f(w_2)=\0$ for $w_1\ne w_2, w_1\in K_f, w_2\in K_f$
and $G_f(w)$ is nonempty if and only if $w$ belongs to the boundary $J_f$ of $K_f$,
the \emph{Julia set} of $f$. The family of edges of all gap-leaves of points $w\in J_f$ is called the \emph{lamination of $f$};
it completely describes $J_f$ and is viewed as a structure that adequately describes combinatorics of $K_f$.

\subsection{The question}
\label{ss:q}
Consider a cubic polynomial $P$.
Suppose that $K_P$ is connected and locally connected, and that all critical points of $P$ belong to $J_P$.
Then either $P$ has distinct critical points $c\ne d$, or $P$ has a unique multiple critical point $c=d$.
Set $C_P:=G_P(c)$, $D_P:=G_P(d)$. Clearly, either $C_P=D_P$, or $C_P$ and $D_P$ are disjoint.
Now, let $P_n$ be a sequence of cubic polynomials with $K_{P_n}$ connected and locally connected and
with critical points $c_n$, $d_n$ defined as above in the Julia set of $P_n$,
so that $C_n:=C_{P_n}$ and $D_n=D_{P_n}$ 
converge (in the sense of the Hausdorff metric) to $C$ and $D$, respectively.

Given points $\ga$ and $\ta$ in $\uc$, let $\ol{\ga\ta}$ be the chord with endpoints $\ga$ and $\ta$.
Assume  that we are given specific arguments $\al$ and $\be$ such that the two chords,
$\oc=\ol{\al\,\al+\frac 13}$ and $\od=\ol{\be\,\be+\frac 13}$, disjoint inside the open
unit disk $\disk$, are contained in $C$ and $D$, resp. We want to
describe possible sets $C$ and $D$, for given $\oc$ and $\od$.
This setup can be justified as follows.
Suppose that for some cubic polynomial $P$ with $K_P$ connected and locally connected
$R_P(\al)$ and $R_P(\al+\frac 13)$ land on the same point while
$R_P(\be)$ and $R_P(\be+\frac 13)$ also land on the same point.
Then necessarily $C=C_P$ and $D=D_P$, that is, there is only one possible choice for $C$ and $D$.
The question remains relevant even in the case when $R_P(\al)$ and $R_P(\al+\frac 13)$ or
$R_P(\be)$ and $R_P(\be+\frac 13)$ cannot land on the same point for dynamical reasons
(such as angles in question being periodic) as these cases are limits of more %adequate
 straightforward cases listed above.

The question addressed in our paper is the following:
given that $0,\frac 13\in C$ and $0,\frac 23\in D$, \emph{what can $C$ and $D$ be}?
Recall that $\R/\Z$ is identified with the unit circle, so that $0$, $\frac 13$, $\frac 23$ are
points of the unit circle with arguments $0$, $\frac{2\pi}3$, $\frac{4\pi}{3}$, respectively,
i.e., the cubic roots of unity. Observe that, as was remarked above, neither rays
$R_P(0)$, $R_P(\frac 13)$, nor rays $R_P(0)$, $R_P(\frac 23)$ can ever have the same landing point,
no matter which $P$ is taken.

\subsection{The answer}
\label{ss:ans}
We first show that either $C=\ol{0\frac 13}$ or $D=\ol{0\frac 23}$;
we consider only the former case as the latter case reduces to the former by
interchanging $C$ with $D$ while simultaneously reflecting the unit circle with respect to its horizontal diameter.
When $C=\ol{0\frac 13}$ is fixed, the set $D$ is determined by $\si_3(D):=\ch(\si_3(D\cap\uc))$;
here, $\ch(X)$ denotes the convex hull of $X$ with respect to the usual real affine structure on $\C=\R^2$
(in figures, the Poincar\'{e} metric on $\disk$ is used instead of the real affine structure to form convex hulls).
Recall that $\si_3(z)=z^3$ is the angle tripling map acting on the unit circle and not defined inside the unit disk;
if $X\subset \uc$ is closed, the set $\si_3(\ch(X))$ is \emph{not} the pointwise image of $\ch(X)$ under $z\mapsto z^3$
but the set $\ch(\si_3(X))$. Now, by our assumption
$\ol{0\frac 23}\subset D$. Hence, given $\si_3(D)$, the set $D\cap\uc$ can be recovered as the set of all
$\si_3$-preimages of the points from $\si_3(D)\cap\uc$ that lie in the bigger circle arc $[\frac 13,1]$ bounded by $\frac 13$ and $1=0$.

Let $T^0$ be the convex hull of the points $0$ and $1/3^n$ for $n=1$, $2$, $\dots$.
Define $T^1$ as the $\si_3$-pullback of $T^0$ in the upper half-disk, i.e., the set $T^1=\ch(T^1\cap\uc)$
such that $\si_3(T^1)=T^0$ and $T^1$ lies non-strictly above the horizontal diameter.
Then $T^0\cap T^1=\ol{0\frac 13}$ and $T^1$ lies non-strictly between $T^0$ and the horizontal diameter.
By induction on $k=1$, $2$, $\dots$, define $T^{k+1}$ as the $\si_3$-pullback of $T^k$ that
lies between $T^k$ and the horizontal diameter.
See Fig. \ref{fig:Tk} for an illustration of the sets $T_k$.

\begin{figure}
  \centering
  \includegraphics[width=.7\textwidth]{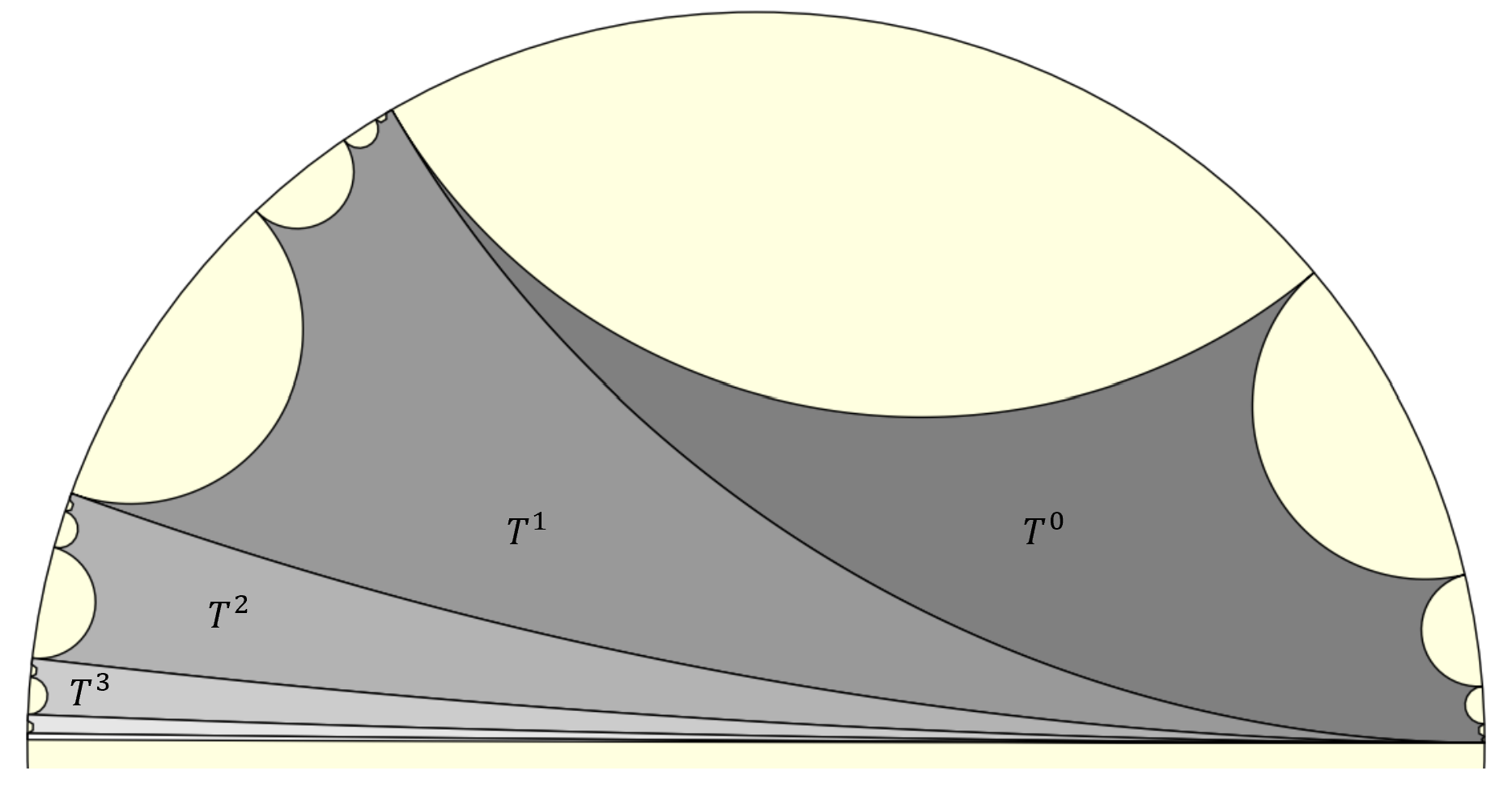}
  \caption{The sets $T^k$.}\label{fig:Tk}
\end{figure}

\begin{thm}
\label{t:main1}
Assuming that $C=\ol{0\frac 13}$ and that $D\ni 0,\frac 23$, the
 only possibilities for $\si_3(D)$ are $T^k$, for some $k=0$, $1$, $\dots$,
 the top edge of $T^{k+1}$, or the single point $\{0\}$; all these possibilities realize.
\end{thm}

A direct description of all possibilities for $D$, rather than for $\si_3(D)$, is as follows.
Consider the convex hull $H^k$ of the full $\si_3|_{[\frac 13,1]}$-preimage of $T^{k-1}\cap\uc$.
This convex set contains $T^k$.
Also, consider the convex hull $Q^k$ of the bottom edge $E$ of $T^{k-1}$ and
the only other edge of $H^k$ that has the same $\si_3$-image as $E$;
the set $Q^k$ thus defined is a quadrilateral.

\begin{cor}
  \label{c:main1}
Under the assumptions of Theorem \ref{t:main1}, the set $D$
can only be $H^k$, or $Q^k$, for some $k>0$, or $\ol{0\frac 23}$.
All these possibilities realize.
\end{cor}

An alternative description of the sets $H^k$ and $Q^k$ together with figures illustrating the picture can be found at the end
of the Introduction.

\subsection{Critical portraits}
Theorem \ref{t:main1} implies a description of limit critical portraits.
A set of pairwise disjoint (except possibly for common endpoints) chords $\ol{\al_1\be_1}$, $\dots$, $\ol{\al_{d-1}\be_{d-1}}$ of $\disk$,
where $\be_i-\al_i\in \frac 1d\Z$, is called a degree $d$ \emph{critical portrait} if the union of these chords does not separate the plane.
There are several cases when a critical portrait can be associated to a polynomial in a natural fashion.

First, assume that a polynomial $f$ has connected and locally connected filled Julia set $K_f$.
Then $\ol{\al_1\be_1}$, $\dots$, $\ol{\al_{d-1}\be_{d-1}}$ is a \emph{critical portrait of a polynomial} $f$
if $R_f(\al_i)$ and $R_f(\be_i)$ land on the same point, for $i=1$, $\dots$, $d-1$
(these landing points are then necessarily critical points of $f$).
This imposes significant restrictions on $f$,
namely, not only $K_f$ must be connected and locally connected but also all critical points of
$f$ lie in $J_f$. Such polynomials form a subset of the set of
\emph{Siegel-dendritic}, or \emph{Side} polynomials \cite{bot26} defined as polynomials that do not have hyperbolic or parabolic
Fatou domains. In the end we obtain a (finite) family of critical portraits associated with a polynomial $f$ where $f$ is a Side polynomial with
connected locally connected Julia set.

A related class of polynomials is that of \emph{dendritic} polynomials, i.e. polynomials with only repelling periodic points.
Dendritic polynomials are Side. On the other hand, unlike polynomials considered in the previous paragraph, dendritic polynomials with
connected Julia set do not necessarily have locally connected Julia sets. However by Kiwi \cite{kiw04} (see also \cite{bco11}) \emph{every dendritic polynomial
with connected Julia set can be naturally associated to critical portraits}. To explain that one needs to use the concept of the
\emph{impression of the external ray}, an important concept from continuum theory \cite{ill25}. Let a dendritic polynomial $P$ of degree $d$
have a connected Julia set $J_P$. Declare two angles $\al, \be$ equivalent if the impressions of their external rays are non-disjoint
and extend this equivalence by transitivity.

By \cite{kiw04} and \cite{bco11} the resulting equivalence relation on the circle is closed, with
finite classes of equivalence.
Moreover, consider convex hulls of these classes of equivalence and call them \emph{gap-leaves}.
The family of edges of all gap-leaves of points $w\in J_f$ is then called the \emph{lamination of $f$}; similar to the above in the locally connected
case it is viewed as a structure adequately describing combinatorics (but not topology!) of $K_f$.
For every gap-leaf the union of the corresponding
impressions is a (possibly degenerate) continuum nowhere dense in $\C$, and the entire $J_f$ is then partitioned into such continua
called \emph{(generalized) impressions (of gap-leaves of $f$)}.

Now, take the generalized impressions of gap-leaves that contain critical points. There are finitely many of them.
Call the corresponding gap-leaves \emph{critical}. Any critical portrait whose chords are contained in the critical
gap-leaves of a dendritic polynomial $f$ can be associated with $f$. Thus, in this case, too, a finite collection of
critical portraits is associated with a dendritic polynomial $f$, and in this case the requirement of having a locally connected Julia set
is not needed.

The situation with dendritic polynomials is discussed in more detail later in Section \ref{s:mot}.

Suppose now that $P_n$ is a sequence of Side cubic polynomials with critical portraits converging to $\{\ol{0\frac 13},\ol{0\frac 23}\}$
(that is, $P_n$ are polynomials that can be associated to critical portraits as described above).
If we choose another sequence of critical portraits
in the critical sets $C_n$, $D_n$, what can these critical portraits converge to?
This is a version of our main question.

More notation is needed to state the answer.
Let $\Delta'$ be the full $\si_3|_{[\frac 13,1]}$-preimage of $\ol{(T^0\cup T^1\cup\dots)}\cap\uc$,
and set $\Delta=\ch(\Delta')$.
Also, let $\tau$ be the reflection of $\uc$ with respect to the horizontal diameter,
that is, $\tau(\al)=-\al$ in terms of the angular coordinate $\al\in\R/\Z$.

\begin{cor}
\label{c:main-po}
Let $P_n$ be a sequence of Side cubic polynomials with whom critical portraits $\{\oc_n,\od_n\}$ can be associated. %with $K(P_n)$ connected and locally connected.
If $\{\oc_n,\od_n\}$ is a critical portrait of $P_n$, and if some critical portraits of $P_n$ converge to $\{\ol{0\frac 13},\ol{0\frac 23}\}$,
 then the only possible limits of $(\oc_n,\od_n)$ are cubic critical portraits contained in $\Delta$ or in $\tau(\Delta)$.
\end{cor}

See Fig. \ref{fig:Del1} for an illustration of Corollary \ref{c:main1}.

\begin{figure}
  \centering
  \includegraphics[width=.6\textwidth]{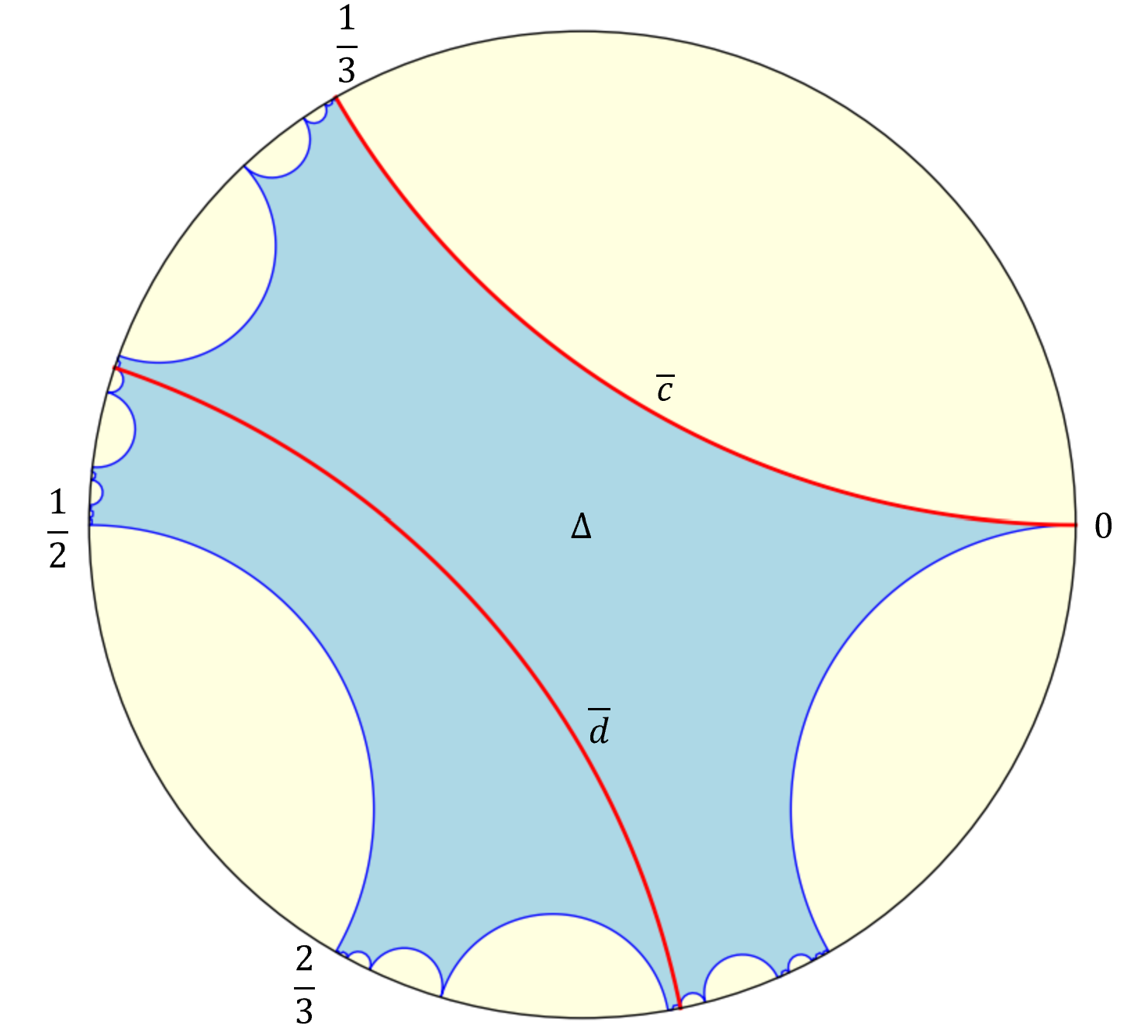}
  \caption{All possibilities for $\{\lim\oc_n,\lim\od_n\}$ are critical portraits contained in $\Delta$ or in $\tau(\Delta)$.
  One such portrait $\{\oc,\od\}$ is shown in the figure.
  }\label{fig:Del1}
\end{figure}

\subsection{Laminations}
\label{ss:intro-lam}
Consider again a polynomial $f$ with $K_f$ connected and locally connected.
Among the gap-leaves of all points of $J_f$, the ones associated with the critical points of $f$ are the most important.
This is in line with a general observation from complex dynamics that the dynamical
behavior of the critical points is largely responsible for the dynamics of all other orbits.
Still, it makes sense to consider all gap-leaves at once; they form structures called \emph{invariant laminations}.
Define the equivalence relation $\sim_f$ on $\uc$ so that classes of $\sim_f$ are precisely $G_f(z)\cap \uc$, for all $z\in J_f$
(the gap-leaf $G_f(z)$ of $z$ was defined in the beginning of the Introduction, before Section \ref{ss:q}).
It can be shown that $f:J_f\to J_f$ is then topologically conjugate to the map induced on $J_{\sim f}:=\uc/\sim_f$ by $\si_d:z\mapsto z^d$.
Form the collection $\lam_f=\lam_{\sim_f}$ of all edges of all $G_f(z)$; this is the \emph{lamination of $f$}.
Elements of $\lam_f$ are called \emph{leaves} while the closures of the complementary components
of the union of all leaves are called \emph{gaps}.

A general concept of \emph{$\si_d$-invariant geodesic laminations} was developed by Thurston in \cite{thu85}.
These are collections of chords that have similar properties to those of $\lam_f$, for $K_f$ connected and locally connected.
One significant advantage of this more general approach is that limits of invariant laminations are also invariant laminations.
However, if $f_n$ is a sequence of polynomials %and $\lam_n=\lam_{f_n}$ are the corresponding laminations,
with connected and locally connected $J_{f_n}$,
then the limit $\lam$ of $\lam_{f_n}$, even if it exists, may not have the form $\lam_f$, for a polynomial $f$.
For example, it may happen that $f_n\to f$ but $K_f$ is not locally connected,
or $K_f$ is locally connected but $\lam_f$ fail to coincide with $\lam$
(the latter may happen when $f$ contains \emph{parabolic periodic points}, i.e., points $z$ with $f^p(z)=z$ and $(f^p)'(z)=1$ for some $p>0$).

Consider again a sequence of cubic polynomials $P_n$ with $K_{P_n}$ connected and locally connected
such that some (suitably chosen) critical portraits of $P_n$ converge to $\{\ol{0\frac 13},\ol{0\frac 23}\}$.
Yet another version of our question is \emph{what can the limit of $\lam_{P_n}$ be}.
An explicit answer to this question is as follows.
Fix some $k=1$, $2$, $\dots$, and consider the convex sets $T^0$, $\dots$, $T^{k-1}$.
Instead of $T^k$, take the convex hull $H^k$ of the full $\si_3|_{[\frac 13,1]}$-preimage of $T^k\cap\uc$.
Observe that the union $(T^0\cup\dots \cup T^{k-1}\cup H^k)\cap \uc$ is a forward $\si_3$-invariant countable
set all of whose complementary components are circle arcs of length $<\frac 13$.
There is therefore a unique $\si_3$-invariant lamination $\lam^k$ that contains $T^0$, $\dots$, $T^{k-1}$, and $H^k$ as gaps.
In this way, we defined a sequence of laminations $\lam^k$.
See Fig. \ref{fig:lams} for an illustration of $\lam^1$ and $\lam^2$.

\begin{figure}
  \centering
  \includegraphics[width=.9\textwidth]{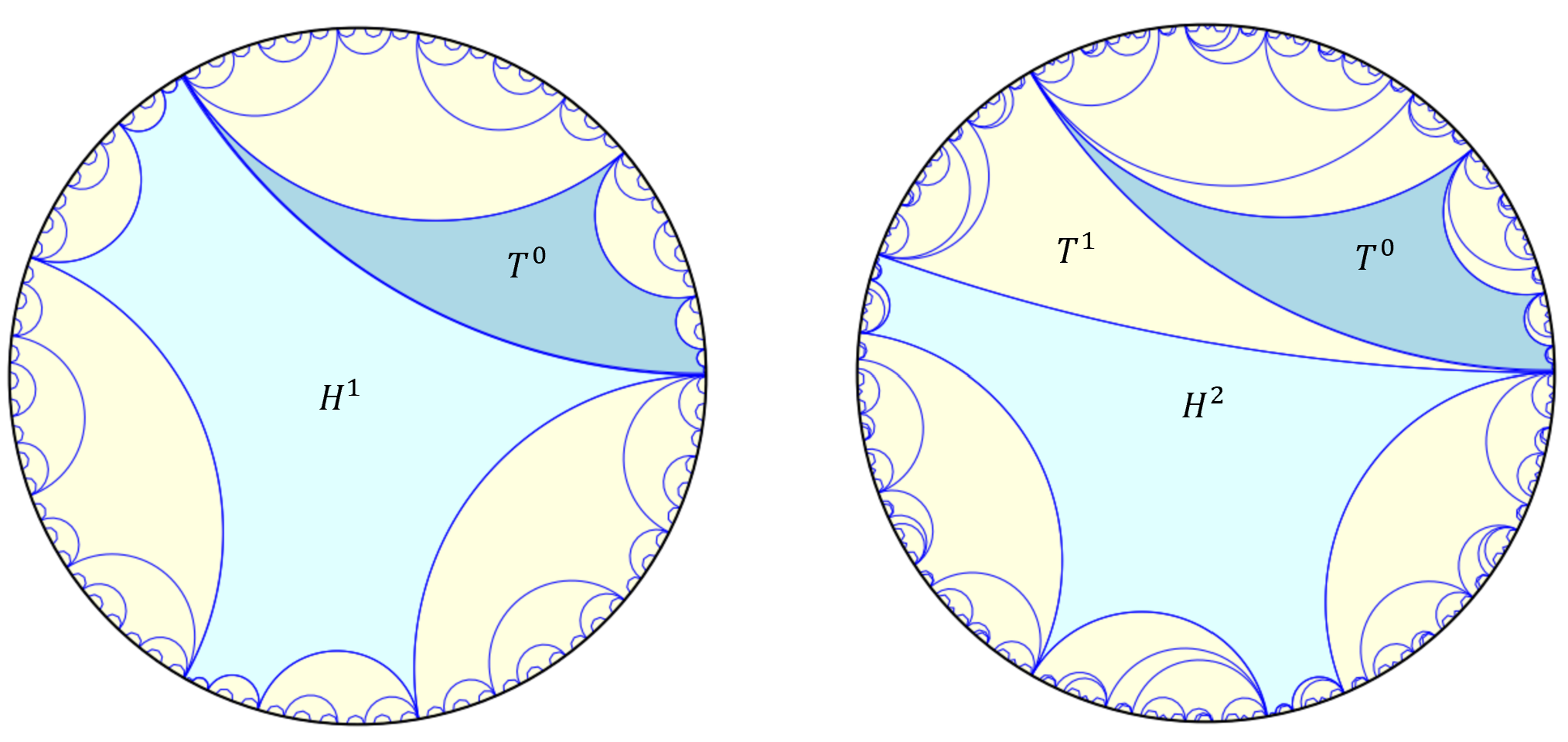}
  \caption{Left: $\lam^1$. Right: $\lam^2$.}\label{fig:lams}
\end{figure}

For $k>1$, one can modify $\lam^k$ as follows.
Insert a critical quadrilateral $Q^{k}$ into $H^{k}$
so that $Q^{k}$ shares an edge with $T^{k-1}$;
there is a unique choice of $Q^{k}$ with this property.
Accordingly, insert eventual pullbacks of $Q^{k}$ into eventual pullbacks of $H^{k}$.
Denote the thus obtained lamination by $\lam^k_\square$. See Figure \ref{fig:q2} for an illustration of $\lam^2_\square$.

\begin{figure}
  \centering
  \includegraphics[width=.9\textwidth]{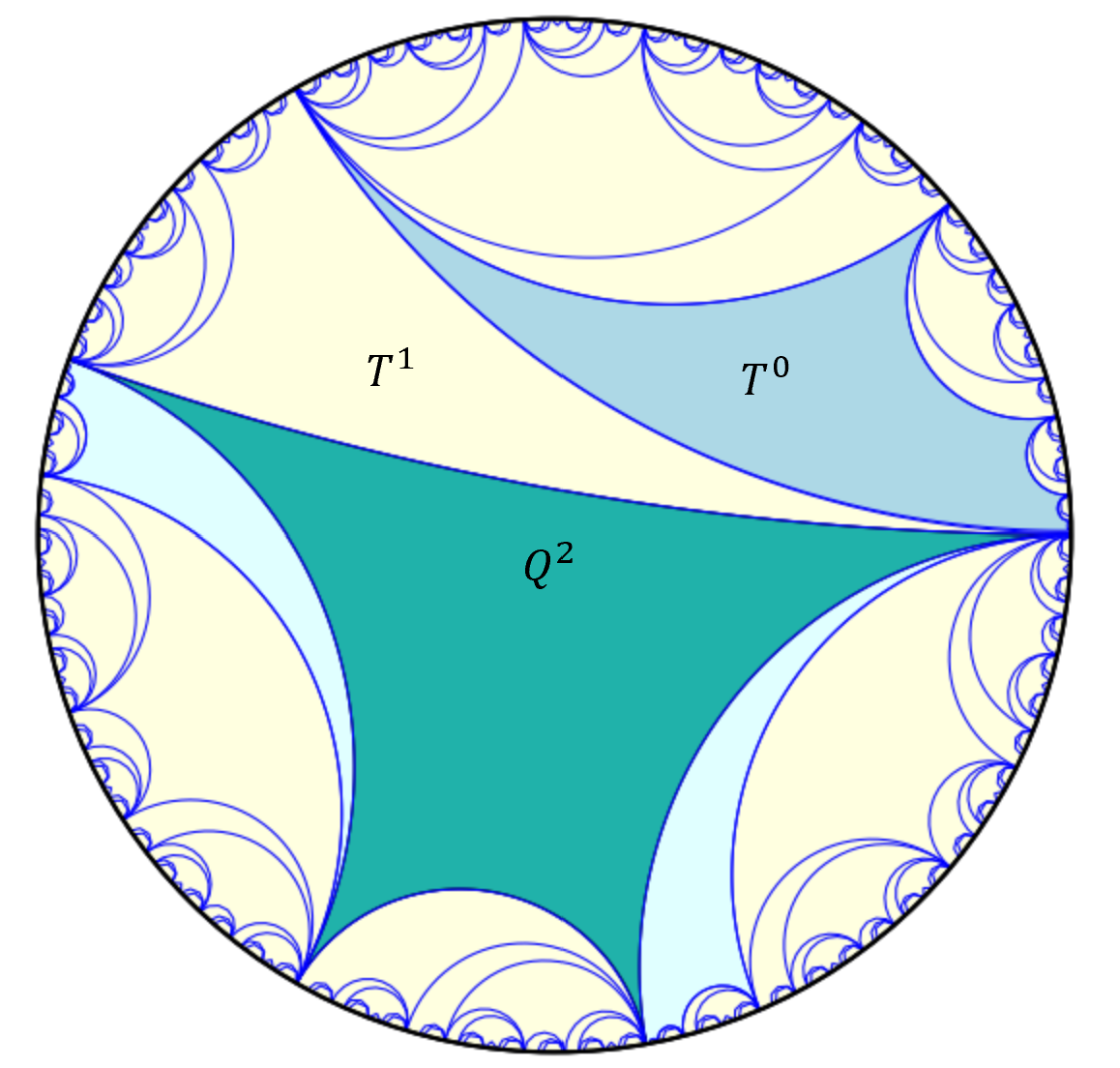}
  \caption{$\lam^2_\square$.}\label{fig:q2}
\end{figure}

\begin{thm}
  \label{t:main-lam}
Suppose that some critical portraits of Side cubic polynomials $P_n$
converge to $\{\ol{0\frac 13},\ol{0\frac 23}\}$ as $n\to\infty$.
If the laminations $\lam_{P_n}$ also converge, then the limit lamination $\lam$ coincides with $\lam^k$, $\lam^k_\square$,
$\tau(\lam^k)$ or $\tau(\lam^k_\square)$, for some $k$, unless $\ol{0\frac 13}$, $\ol{0\frac 23}\in\lam$.
All these options realize.
\end{thm}

All the results stated in the Introduction can be reformulated (and generalized) in the setting of
laminations and critical portraits, without any reference to polynomials.
This will be done later in the text.

The plan of the paper is as follows. In Section \ref{s:mot} we describe our motivation considering known quadratic results and
describing advances made in the cubic case so far. In Section \ref{s:lclam} we
discuss basic facts about laminations. In Section \ref{s:advance} we supply advanced results on laminations. Finally, in Section \ref{s:main} we
prove the main results of the paper.

\section{Motivation: a quest for a model}
\label{s:mot}
In this section, some basic terminology and concepts from holomorphic polynomial dynamics are assumed;
 detailed definitions concerning invariant laminations are given later.

\subsection{The quadratic case: the statement of the problem}\label{ss:quad}

An important problem in complex dynamics is to construct a model of the %degree $d$
 connectedness locus in the space of all degree $d$ polynomials.
In the quadratic case, this was done
by Thurston in the 80-s in his seminal preprint that appeared in print much later \cite{thu85}.
Thurston related complex polynomials to \emph{laminations}.

Thus, constructing a model for the quadratic connectedness locus was reduced
to constructing a model for the space of quadratic invariant laminations.

Observe that, in the quadratic case, critical portraits are simply diameters, and the space of diameters is, topologically, a circle.
Essentially, Thurston associated each quadratic lamination $\lam$ with diameters compatible with $\lam$ and, in the periodic case, satisfying
some additional properties. In this way he defined an equivalence relation on the space of
diameters which has the following properties:

\begin{enumerate}
\item 	the equivalence relation is closed,
\item 	classes are finite, and
\item	each class corresponds to a unique q-lamination and every lamination corresponds to a unique class.
\end{enumerate}

Hence the associated quotient space of the circle is a model of the space of all quadratic q-laminations.

\subsection{The cubic case: the statement of the problem}\label{ss:cubic}

The  situation for degree 3 is much more complicated, except that there is a transparent
description of the special class of cubic dendritic laminations.
Let us state some facts that follow from \cite{bopt19}.
Consider only cubic q-laminations that have no infinite gaps (such q-laminations are said to be
\emph{dendritic} regardless of the degree).
For brevity let us call critical portraits of dendritic q-laminations \emph{dendritic critical portraits}.
With each dendritic q-lamination $\lam$, associate the \emph{collection of all critical portraits compatible with $\lam$}
called the \emph{primary tag} of $\lam$. This \emph{partitions} the space of all
dendritic critical portraits into primary tags; equivalently, a dendritic critical portrait defines
a unique dendritic q-lamination with which it is compatible. We call this the \emph{primary partition (of the space of dendritic critical portraits)}.

The primary partition has nice topological properties
similar to those listed above in the quadratic case.
Consider the equivalence relation $\approx$ on the space of all dendritic critical portraits with classes of equivalence being primary tags of dendritic q-laminations.
If $\Po'_i\approx \Po''_i$ are equivalent critical portraits, $\Po'_i\to \Po'$, and
$\Po''_i\to \Po''$ where $\Po'$ is a dendritic critical portrait, then $\Po''$ is a critical portrait compatible with the same dendritic lamination as $\Po'$, and so $\Po'\approx \Po''$. Thus, the primary partition is an upper semicontinuous partition of the non-compact space of all dendritic critical portraits.

Since dendritic critical portraits are dense in the space of all critical portraits,
this naturally leads to the next problem.

\begin{mainpro}
Describe the partition of the space of all critical portraits which is the  closure of the primary partition, i.e. the finest closed equivalence relation on the space of all critical portraits that extends the primary equivalence relation.
\end{mainpro}

Solving it will yield a natural combinatorial (one can also say ``laminational'') model of the space of all q-laminations and, eventually,
to a model of the associated connectedness locus of polynomials.

In practical terms this means that we should consider various (not necessarily dendritic!) limit critical portraits
$\Po$ and $\Po'$ and assume that there are equivalent dendritic critical portraits $\Po_i\to \Po$ and $\Po'_i\to \Po'$. As the first step in the direction of solving the Main Problem we can try to answer the following question: given $\Po$, what can $\Po'$ be?

We analyze this problem
for the ``worst'' non-dendritic critical portrait, in a way opposite (in terms of its properties) to the dendritic critical portraits.
While the dendritic critical portraits are all such that no critical chord of a critical portrait has a periodic endpoint, in this paper we consider
the critical portrait $\Po=(\ol{0 \frac13}, \ol{0 \frac23})$ whose critical leaves share a common fixed endpoint. Our analysis leads to an explicit description
of the family of critical portraits $\Po'$ which is a special closed countable family of critical
portraits. We will also describe the associated q-laminations.

\subsection{The quadratic case: the details}\label{ss:quadetail}

Let us now go back and describe the quadratic case in more detail. Because laminations are closely
related to  polynomials, Thurston's model also provides a model for the connectedness locus of quadratic polynomials.
A major problem in complex dynamics, the so-called MLC conjecture, is to prove that Thurston's model is homeomorphic to the boundary of
the Mandelbrot set $\Mc_2$ (equivalently, to
verify that $\Mc_2$ is locally connected).

In the quadratic case the space of polynomials consists of complex polynomials of the form $P_c(z)=z^2+c$ forming the \emph{quadratic family}.
The fact that it is real 2-dimensional plays a significant role in \cite{thu85}. In particular, there are countably many parabolic polynomials
in the quadratic family, they all have distinct dynamical properties and belong to the boundary of $\Mc_2$. Thus,
no 2-dimensional disks inside $\Mc_2$ consist of polynomials with parabolic dynamics. The construction from \cite{thu85},
that essentially deals with the boundary of $\Mc_2$ (i.e., with the part of $\Mc_2$ ``visible'' from infinity),
provides a model for the entire Mandelbrot set.

More specifically, Thurston associates with every polynomial $P_c$ with locally connected Julia set $J_c$
the set of arguments of external rays landing at the critical value $c$ of $P_c$. A less elegant but essentially equivalent construction
can be implemented if one associates with each quadratic polynomial with locally connected Julia set the set %%of pairs
of arguments
of rays landing on its critical point. % and having the same image ray.
This can be extended onto all polynomials $P_c\in \Mc_2$.

As a combinatorial analog of polynomials, Thurston introduced the concepts of an \emph{invariant lamination} and of
a \emph{topological polynomial}. Consider a polynomial
$P$ with locally connected Julia set. Take the Riemann map for the basin of attraction of infinity and the external rays defined by this map. Declare angles $\al, \be$
equivalent if the external rays with arguments $\al, \be$ land on the same point. This
equivalence relation is closed, all classes of equivalence are finite, and their convex hulls are pairwise disjoint.
The family of edges of convex hulls of all these equivalence classes is called the \emph{q-lamination (of $P$)}.

Q-laminations can be defined with no regard to polynomials, i.e. by listing their geometric properties.
Thurston did that
in \cite{thu85} but the definitions he gave were slightly less restrictive. This led to the concept of a \emph{lamination}.
More precisely, two segments in $\R^2$ are \emph{unlinked} if their
union does not have the shape of the letter $X$. For a chord $\ol{xy}\subset \cdisk$ set $\si_d(\ol{xy})=\ol{\si_d(x) \si_d(y)}$ (here $\si_d$ is
the $d$-tupling map of the unit circle $\uc$ to itself). A family of pairwise unlinked
chords is said to be a \emph{$\si_d$-invariant lamination} if $\si_d$ acts on them so that certain dynamical properties are satisfied
(these properties mimic those of invariant q-laminations).
These chords are called \emph{leaves (of the lamination)} while the closures of
the complementary components of $\cdisk$ to the union of all the leaves are called  \emph{gaps
(of the lamination)}. A gap $G$ is said to be  \emph{finite/infinite} if so is $G\cap \uc$.

\subsection{The dendritic case}\label{ss:dendr}

For a polynomial $g$ with non-locally connected Julia set $J_g$,
the above construction does not yield any q-la\-mi\-na\-tion.
However, there is a way around when $g$ has only repelling cycles.
In the Introduction we considered such polynomials and briefly explained how in that case a q-lamination
can be associated to a polynomial with connected but not locally connected Julia set; our presentation was based upon the results of
\cite{kiw04} (see also \cite{bco11}). For the sake of completeness we present here an equivalent approach, also based upon
\cite{kiw04} and \cite{bco11}. Namely,
in that case a result of Jan Kiwi \cite{kiw04} (see also \cite{bco11}) provides for the following construction
(observe that the result and the associated construction apply to polynomials of any degree $d>1$).
Take all (pre)periodic points $z$ of $g$ and consider the corresponding gap-leaves $G_g(z)$
(recall that $G_g(z)$ is the convex hull of the set of arguments of all external rays landing on $z$).
Consider the edges of all these gap-leaves.
The closure of the just constructed family of chords is an invariant lamination $\lam_g$.
Declare points $u, v\in \uc$
\emph{equivalent in the sense of $g$} if a finite concatenation of leaves of $\lam_g$ connects $u$ and $v$; denote this relation $\sim_g$.
The equivalence relation $\sim_g$ is closed.
The quotient space $\uc/\sim_g=J_{\sim_g}$ is a %metrizable
continuum called the \emph{topological Julia set}
on which $\si_3$  acts.  The  map $F_g$, induced on $J_{\sim_g}$ by $\si_d$,  is called the \emph{topological polynomial induced by $g$}.
By Kiwi \cite{kiw04}, there exists a monotone map $\psi_g:J_g\to \uc/\sim_g$ semiconjugating $g|_{J_g}$ and $F_g$.
The topological Julia set $J_g$ will be a \emph{dendrite} (i.e., a locally connected continuum with no subsets that are Jordan curves)
which is why such laminations $\lam$ (and the corresponding polynomials) will be called \emph{dendritic}.

However if a polynomial is not dendritic and has a non-locally connected Julia set its association with a combinatorial
structure (such as lamination) is more complicated. An indirect way of doing so can be based upon the fact that polynomials
exist not in vacuum but in spaces of polynomials. Therefore a natural way of handling this problem is to approximate a polynomial $P$
by polynomials with locally connected Julia sets or dendritic polynomials, and then to associate to $P$ the Hausdorff limit
of the corresponding sequence of q-laminations. Such limits have properties similar to those of q-laminations
but not exactly the same. Thurston's definition of a lamination aims at taking care of this issue using geometric
and dynamic properties of leaves of laminations and relaxing properties of q-laminations to absorb limits of q-laminations as well.

\subsection{The cubic case: the details}\label{ss:cubic-detail}

In this paper we are interested in the cubic case. Call a chord $\ol{ab}$ \emph{critical} if $\si_3(a)=\si_3(b)$. Call a pair of
unlinked \emph{critical} chords a \emph{critical portrait}. Given a dendritic cubic (i.e., $\si_3$-invariant) lamination $\lam$, call a
collection of chords $\Lf$ \emph{compatible}
with $\lam$ if no chord $\ell\in \Lf$ crosses leaves of $\lam$.
As in the quadratic case, we tag a dendritic lamination $\lam$ with its \emph{critical tag}, defined as
the collection of all critical portraits compatible with $\lam$. If $f$ is a cubic polynomial with only repelling cycles then, by \cite{kiw04},
we associate $f$ with the invariant dendritic lamination $\lam_f$ and also with the critical tags of $\lam_f$.
We call a critical portrait $\Po$ a \emph{Side critical portrait} if all critical leaves in $\Po$ have
non-periodic endpoints. A q-lamination $\lam$ is called a \emph{Side q-lamination} (for Siegel/dendritic)
if all critical portraits compatible with it are Side critical portraits.

Let us now give an overview of \cite{bopt19}. 
Critical tags of distinct dendritic cubic laminations
are pairwise disjoint (do not share critical portraits).
Suppose that $\Po_\circ$ is a critical portrait compatible with a cubic dendritic lamination $\lam$
 and that $\Po_\circ$ is the limit of a sequence of critical portraits of dendritic laminations $\lam_n$.
Let $\Po$ be the limit of some other sequence of critical portraits of $\lam_n$.
One of the results of \cite{bopt19} claims that then $\Po$ is compatible with $\lam$.
In other words, the family of critical tags of distinct dendritic cubic laminations is \emph{upper semicontinuous in itself}
which makes the map from cubic dendritic laminations to their critical tags a continuous map of non-compact spaces.
This, together with some geometric considerations,
 allowed us to construct a model of the space of all cubic dendritic polynomials \cite{bopt19}.

As a step towards extending this construction onto the entire cubic connectedness locus, in this paper we consider the
case $\Po_\circ=(\ol{0 \frac13}, \ol{0 \frac23})$,
 which can be viewed as an extreme opposite to the dendritic critical portraits considered in \cite{bopt19}.
It turns out that for $\Po_\circ=(\ol{0 \frac13}, \ol{0 \frac23})$ the corresponding set of limit critical portraits $\Po$ is countable, closed
and is associated with a specific countable family of laminations which we explicitly describe
 (see Corollary \ref{c:main-po} and Theorem \ref{t:main-lam}).

The present paper extends some results of \cite{bopt19} as it studies the convergence of Side critical portraits under new assumptions.
We believe that the description of limit critical portraits obtained in 
 Corollary \ref{c:main-po}
 can be useful for solving the problem of constructing
a combinatorial model of the cubic connectedness locus. This serves as our motivation.

\section{Laminational equivalence relations} %and our main theorem}
\label{s:lclam}
%Set $\uc=\{z\in\C\,|\,|z|=1\}$.
Let $\M_d$ denote the connectedness locus in the space of all complex degree $d$ monic centered polynomials.
We parameterize the external
rays of a polynomial $f\in\M_d$ by %\emph{angles}, i.e.,
elements of $\R/\Z$ (so, the full angle is $1$ rather than $2\pi$).
The external ray of argument $\theta\in\R/\Z$ is denoted by $R_f(\theta)$;
 note that $f$ maps $R_f(\theta)$ to $R_f(d\theta)$.
Let $\si_d:\uc\to \uc$ be the self-mapping of $\uc\subset \C$ that takes $z$ to $z^d$
(if we represent $\uc$ as $\R/\Z$ then $\si_d(t)=d\cdot t\, \mod\, 1$).
A \emph{chord} $\ol{ab}$ is a closed straight line segment connecting points $a$, $b$
of $\uc$, often represented by their arguments
(thus, $\ol{\frac 13\frac 23}$ is the chord connecting the points $e^{2\pi i/3}$ and $e^{4\pi i/3}$).
Two chords \emph{cross} if their endpoints alternate on $\uc$; an important class of
families of chords are \emph{prelaminations}, i.e. families of chords that pairwise do not cross.
A closed prelamination is called a \emph{lamination}.
Chords from a prelamination $\lam$ are called \emph{leaves of $\lam$}.

Convex hulls of closed nowhere dense subsets $X$ of $\uc$ are called \emph{polygons} and are denoted by $\ch(X)$ (without assuming
that the cardinality of $X$ is finite). Polygons are \emph{finite, infinite, countable, uncountable} depending on the cardinality
of the set $X$. If $X$ consists of more than two points, all edges of $\ch(X)$ are leaves of a lamination $\lam$, and no leaf
intersects the interior of $\ch(X)$, then $\ch(X)$ is said to be a \emph{gap} (of $\lam$). 
Oftentimes gaps are denoted by $G$; then the set $G\cap \uc$ 
is called the \emph{basis} of $G$.
Define the $\si_d$-image of a chord $\ol{ab}$ as the chord $\ol{\si_d(a) \si_d(b)}$, i.e. through the action of $\si_d$
on $a$ and $b$ (when we say that we apply $\si_d$ to a chord $\ol{ab}$ we apply it to the endpoints of $\ol{ab}$).

For $f\in\M_d$, let $\psi$ be the Riemann map $\psi:\mathbb C\setminus \overline{\mathbb D}\to\mathbb C\setminus K_f$ with $\psi'(\infty)>0$
so that $f\circ \psi(z)= \psi(z^d)$.
For $f\in\M_d$ with locally connected $J_f$, the map $\psi$ can be continuously extended over the boundary of $\mathbb D$
(i.e., the unit circle $\uc$) so that for the extended map $\ol{\psi}$ we can define $\ol{\psi}(e^{2\pi i\theta})$,
the landing point of $R_f(\theta)$;  then $\ol{\psi}:\uc\to J_f$ is a semi-conjugacy
between $\si_d:\uc\to\uc$ and $f:J_f\to J_f$
called the \emph{Caratheodory loop}.
Recall that $J_f$ is locally connected if and only if so is $K_f$.
Define an equivalence relation $\sim_f$ on $\uc$ as follows: $x \sim_f y$ if and only if $\ol{\psi}(x)=\ol{\psi}(y)$
and call $\sim_f$ the \emph{laminational equivalence relation (generated by $f$)}.
The relation $\sim_f$ is $\si_d$-invariant and $\sim_f$-classes --- which evidently coincide with
 $G'_f(z)$ for various $z\in J_f$ (see the Introduction) --- have pairwise disjoint convex hulls.
The quotient space $\uc/\sim_f=J_{\sim_f}$ is called a \emph{topological Julia set}.
Clearly, $J_{\sim_f}$ is homeomorphic to $J_f$.
The map $F_{\sim_f}: J_{\sim_f}\rightarrow J_{\sim_f}$, induced by $\sigma_d$ and
called a \emph{topological polynomial}, is topologically conjugate to $f|_{J_f}$.
See Fig. \ref{fig:bas} for an illustration of these notions.

\begin{figure}[!htb]
    \centering
    \begin{minipage}{0.5\textwidth}
        \centering
        \includegraphics[height=0.2\textheight]{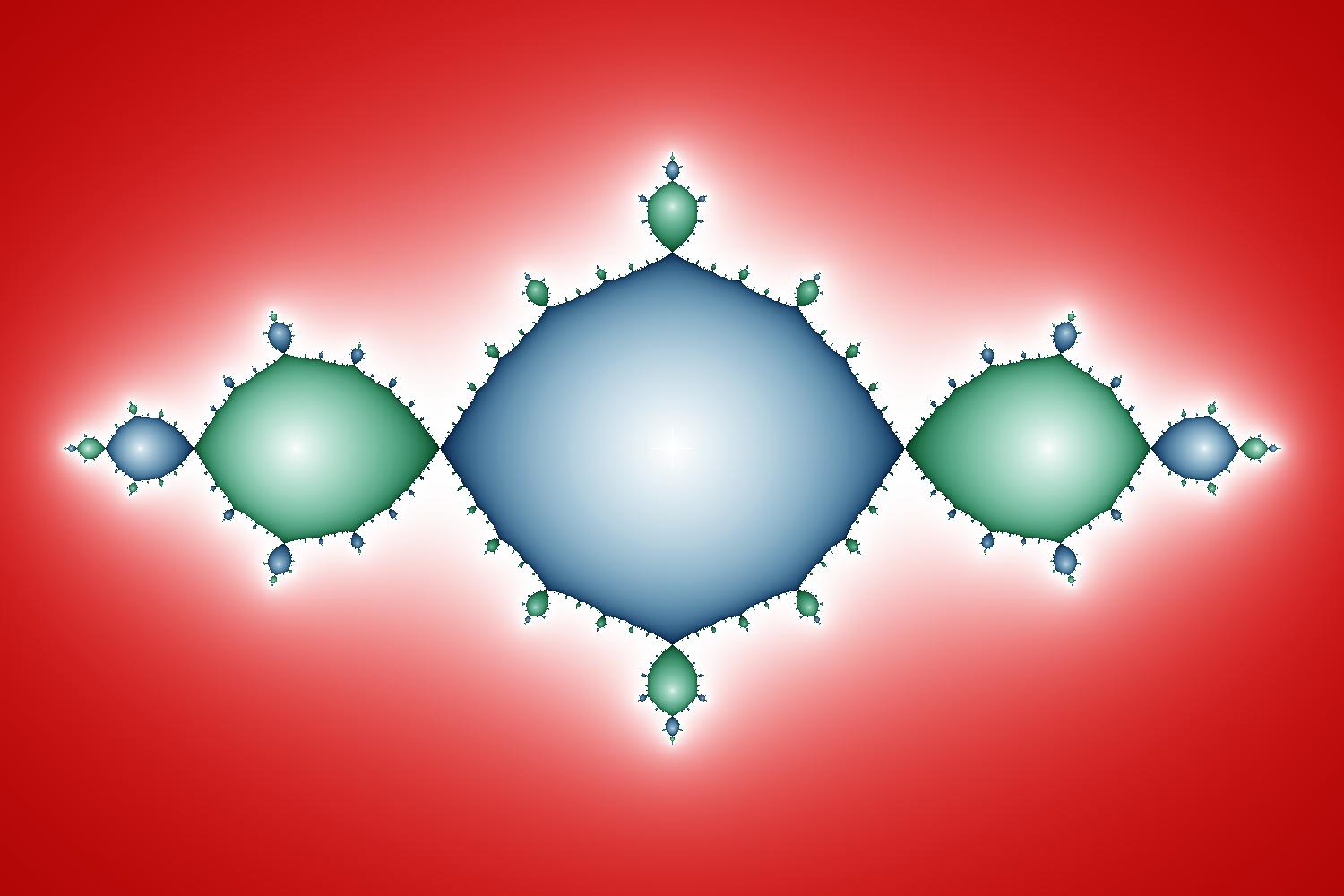}
        %\label{fig:basil}
    \end{minipage}%
    \begin{minipage}{0.5\textwidth}
        \centering
        \includegraphics[height=0.2\textheight]{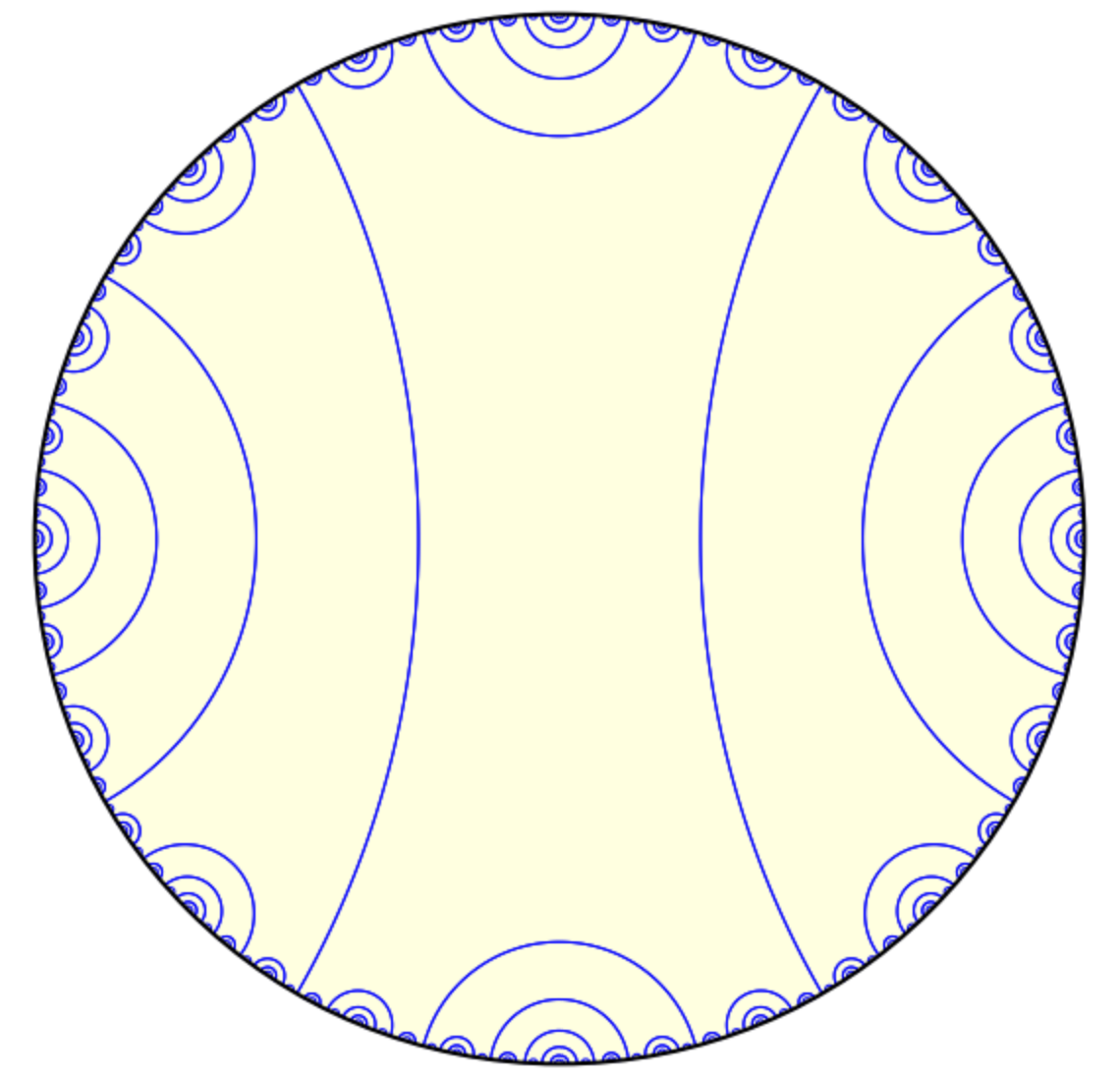}%{lam_basilica.jpg}
        %\caption{The lamination for the Julia set of $z^2-1$}
        %\label{fig:qml}
    \end{minipage}
\label{fig:bas}
 \caption{Left: the Julia set of $f(z)=z^2-1$ (so-called ``basilica'').
 Right: the corresponding $\si_2$-invariant lamination
 ($\sim_f$-equivalent points are connected by arcs inside the disk).}
\end{figure}

Equivalence relations analogous to $\sim_f$ can be introduced
with no reference to polynomials (see \cite{hubbdoua85, thu85}, see also
\cite{bl02}). Let $\sim$ be an
equivalence relation on $\uc$. Equivalence classes of $\sim$ will
be called \emph{($\sim$-)classes}.

\begin{dfn}\label{d:lameq}
An equivalence relation $\sim$ on $\uc$ is a \emph{($\si_d$-)invariant
laminational equivalence relation} if it is:

\begin{enumerate}
\item[(E1)] 
\emph{closed}: the graph of $\sim$ is a closed
set in $\ucirc \times \ucirc$;

\item[(E2)] 
\emph{unlinked}: if $\g_1$ and $\g_2$ are distinct $\sim$-classes,
then their convex hulls $\ch(\g_1), \ch(\g_2)$ in the unit disk $\bbd$
are disjoint;

\item[(E3)] \emph{finite}: all $\sim$-classes are finite;

\item[(D1)] 
\emph{forward invariant}: for a class $\g$,
the set $\si_d(\g)$ is a class too;

\item[(D2)] 
\emph{backward invariant}: for a class $\g$,
its preimage $\si_d^{-1}(\g)=\{x\in \ucirc: \si_d(x)\in \g\}$ is a
union of classes;

\item[(D3)] \emph{orientation preserving}: for any $\sim$-class $\g$ with more than two points, the
map $\si_d|_{\g}: \g\to \si_d(\g)$ is a \emph{covering map with
positive orientation}, i.e., for every connected component $(s, t)$ of
$\ucirc\setminus \g$ either $\si_d(s)=\si_d(t)$ or  the arc in the circle $(\si_d(s), \si_d(t))$ is a
connected component of $\ucirc\setminus \si_d(\g)$.
\end{enumerate}

Here, conditions (E1) -- (E3) are requirements on the  \textbf{E}qui\-va\-lence relation, while
conditions (D1) -- (D3) deal also with the \textbf{D}ynamics of $\si_d$.
Note that (D1) implies (D2).
\end{dfn}

(Invariant) laminational equivalence relations have visual counterparts (in what follows the degree $d\ge 2$ is fixed).
A formal definition of q-laminations (also called \emph{clean laminations} in the language of Thurston) is given below.
The letter ``q'' is added to emphasize that such a lamination corresponds to an invariant laminational
e\textbf{q}uivalence relation on the unit circle.

\begin{dfn}\label{d:q}
Let $\sim$ be an (invariant) laminational equivalence relation. Then the family of all edges of the convex hulls of $\sim$-classes,
together with all points of $\uc$, is called an \emph{(invariant) q-lamination (generated by $\sim$)} and is
denoted by $\lam_\sim$; chords in $\lam_\sim$ are said to be the \emph{leaves} of $\lam_\sim$.
\end{dfn}

Since q-laminations come from laminational equivalence relations we will denote them $\lam_\sim$ indicating that the q-lamination $\lam_\sim$
is generated by a laminational equivalence $\sim$. We will also denote by $\opsi_\sim$ the map from $\uc$ to $J_\sim=\uc/\sim$ that semiconjugates
$\si_d$ and the topological polynomial $F_\sim:J_\sim\to J_\sim$.

\begin{dfn}\label{d:siegel}
A periodic gap $G$ of period $n$ of a q-lamination $\lam_\sim$ is said to be a periodic
\emph{Siegel} gap if $\opsi_\sim$ maps $G\cap \uc$ to a Jordan curve $T\subset J_\sim$ and
semiconjugates $\si_d^n|_{G\cap \uc}$ and $F^n_\sim|_T$ %acting on $T$ as an irrational rotation.
with the latter being conjugate to an irrational rotation of the circle.
\end{dfn}

Now we introduce a new class of laminations and critical portraits. %It is convenient to consider laminations that admit only finite or Siegel periodic gaps.

\begin{dfn}[Side q-laminations]\label{d:side} A q-lamination without infinite gaps is said to be \emph{dendritic}.
A q-lamination with a  Siegel gap is said to be \emph{Siegel}. A q-lamination whose all gaps are finite or Siegel
is said to be \emph{\textbf{Si-}egel-\textbf{de-}ndritic}, or simply a \emph{Side} q-lamination.
A \emph{Side} critical portrait is a critical portrait whose chords have
non-periodic endpoints.
\end{dfn}

It is easy to see \cite{bot26} that a lamination is Side if and only if it is incompatible with any critical chord that has a
periodic endpoint.

\section{Advanced facts about laminations}\label{s:advance}

Here we supply more information about laminations.

\subsection{Thurston invariant laminations}
Given a collection $\Xc$ of sets,
denote by $\bigcup \Xc$ the union of all these sets.

\begin{dfn}[Laminations]\label{d:geolam}
Two distinct chords \emph{cross}, or are \emph{lin\-ked}, if they intersect
inside $\disk$, equivalently, if their endpoints are all distinct and alternate on $\uc$.
A \emph{prelamination} is a family of chords $\lam$ cal\-led
\emph{leaves} such that distinct leaves are unlinked and all points of
$\uc$ are (degenerate) leaves. If the set
$\bigcup\lam$ is compact, then $\lam$ is called a
\emph{lamination}. Two (pre)laminations are \emph{compatible} if their leaves do not cross.
Thus, the union of two compatible (pre)laminations is a (pre)lamination.
\end{dfn}

A set, consisting of one chord and all individual points of $\uc$,
is a prelamination. This prelamination
is compatible with a prelamination if its non-degenerate chord does not cross leaves of that prelamination.
Abusing the language, we will say in this case that \emph{the given chord is compatible with a prelamination}.
From now on, $\lam$ always denotes a prelamination.

\begin{dfn}[Gaps and edges]\label{d:gaps} \emph{Gaps} of a lamination $\lam$ are the closures of
the components of $\disk\sm \bigcup\lam$. A gap $G$ is \emph{countable $($finite,
un\-coun\-table$)$} if $G\cap\uc$ is countable infinite (finite, un\-coun\-table).
Uncountable gaps are called \emph{Fatou} gaps.
\end{dfn}

Convergence of (pre)laminations $\lam_i$ to a set of chords $\mathcal E$
is understood as the Hausdorff metric convergence of leaves of $\lam_i$ to chords from $\mathcal E$;
all Hausdorff metric limits of sequences of chords from $\lam_i$ form $\mathcal E$.
Evidently, $\mathcal E$ is a lamination. A lamination $\lam$ is \emph{nonempty} if it
has nondegenerate leaves and \emph{empty} otherwise
(the empty lamination is denoted by $\lam_\0$; note that it is not the empty set
as it contains all points of $\uc$). Say that $\lam$ is \emph{countable}
if it has countably many nondegenerate leaves and \emph{uncountable}
otherwise; $\lam$ is \emph{perfect} if it has no isolated leaves.
In this paper by \emph{countable} we always mean \emph{infinite countable}.

If $G\subset\cdisk$ is the convex hull of $G\cap\uc$, let $\si_d(G)$
be the convex hull of $\si_d(G\cap\uc)$. This is not the action of a map from
$\cdisk$ to $\cdisk$ restricted on the set $G$. Rather, it is a map from the family of
convex hulls of closed subsets of $\uc$ to itself.
The most important particular case here is when $G=\ol{ab}$ is a chord.
In that case, $\si_d(G)=\si_d(\ol{ab})=\ol{\si_d(a)\si_d(b)}$,
and we call $\ol{ab}$ a \emph{pullback chord} of $\ol{\si_d(a)\si_d(b)}$.
However in some cases it is useful to consider a special extension of $\si_d$ onto $\cdisk$.
The map $\si_d$ can be extended continuously over $\bigcup\lam$
so that it acts linearly on every leaf of $\lam$.
Denote this extended map by $\si_d$, too.

A map $m:X\to Y$ of a topological space $X$ to a topological space $Y$ is \emph{monotone} if for each $y\in Y$, $m^{-1}(y)$ is connected.

\begin{dfn}[Thurston Invariant Lamination \cite{thu85}] \label{dfn-Thurston}
A la\-mi\-na\-tion $\mathcal{L}$ is {\em Thurston ($\si_d$)-invariant} if it
satisfies the following conditions.

\begin{enumerate}

\item Forward $d$-invariance: for any leaf $\ell=\overline{pq} \in
    \mathcal{L}$, either $\sigma_d(p) = \sigma_d(q)$, or
    $\overline{\sigma_d(p)\sigma_d(q)}=\si_d(\ell) \in \mathcal{L}$.

\item Backward invariance: for any leaf $\overline{pq} \in
    \mathcal{L}$, there exists a collection of $d$ {\bf disjoint}
    leaves in $\mathcal{L}$ (this collection of leaves may not be
    unique), each joining a pre-image of $p$ to a pre-image of $q$.

\item Gap invariance: For any gap $G$, the convex hull $H$ of $\si_d(G\cap\uc)$
 is a gap, a leaf, or a single point (of $\uc$). If $H$ is a gap, $\si_d|_{\bd(G)}:\bd(G)\to\bd(H)$ maps as the
composition of a monotone map and a covering map to the boundary of
the image gap, with positive orientation (the image of a point
moving clockwise around $\bd(G)$ must move clockwise around the
image $\bd(H)$ of $G$).
\end{enumerate}

\end{dfn}

By Lemma 2.1 \cite{bl02}, \emph{any q-lamination is Thurston invariant}.

In what follows by $[x, y], (x, y), \dots$ we mean the circle arc of the corresponding type
with endpoints $x, y\in \uc$ such that the movement along $\uc$ inside the arc from $x$ to $y$ is in the positive direction
(i.e., counterclockwise).

\begin{dfn}[Holes]\label{d:hole} Let $G$ be a gap of a Thurston invariant lamination. Let $\ell=\ol{ab}$ such that the motion
from $a$ to $b$ in the arc $(a, b)$ is counterclockwise. Then the arc $(a, b)$ is called a \emph{hole} of $G$.
By Definition \ref{dfn-Thurston}(3), any hole $(a, b)$ of a gap $G$ of a Thurston-invariant lamination either $\si_d(a)=\si_d(b)$ or
the positively oriented arc $(\si_d(a), \si_d(b))$ is a hole of $\si_d(G)$.
\end{dfn}

\subsection{Sibling invariance and gaps}
We will work with so-called \emph{sibling
($\si_d$-in\-va\-riant) laminations}. They form a closed subspace
of the space of all $\si_d$-invariant laminations, which still
contains all q-la\-mi\-na\-tions (in other words, q-laminations and
all their limits are sibling $\si_d$-invariant). Since our main
interest lies in studying q-laminations and their limits, it will
be more convenient to work with sibling $\si_d$-invariant
laminations than with $\si_d$-invariant laminations in the
sense of Thurston. Other advantages of working with sibling
$\si_d$-invariant laminations are that they are defined through
properties of their leaves (gaps are not involved in the definition)
and that the space of all of them is smaller (and hence easier to
deal with) than the space of all Thurston $\si_d$-invariant laminations.

\begin{dfn}\cite[Definition 3.1]{bmov13}\label{d:siblinv}
A lamination $\lam$ is \emph{sibling ($\si_d$-)invariant}
provided:
\begin{enumerate}
\item for each $\ell\in\lam$, we have $\si_d(\ell)\in\lam$,
\item \label{2}for each $\ell\in\lam$ there exists $\ell'\in\lam$ so that $\si_d(\ell')=\ell$.
\item \label{3} for each $\ell\in\lam$ so that $\si_d(\ell)=\ell'$
    is a non-degenerate leaf, there exist d {\bf disjoint} leaves
    $\ell_1,\dots,\ell_d$ in $\lam$ so that $\ell=\ell_1$ and
    $\si_d(\ell_i)=\ell'$ for all $i=1,\dots,d$.
\end{enumerate}
\end{dfn}

Let us list a few properties of sibling $\si_d$-invariant
laminations.

\begin{thm}\cite[Theorem 3.2, Lemma 3.20 and Theorem 3.23]{bmov13}\label{t:siblinv} Sibling $\si_d$-invariant
laminations are invariant in the sense of Thurston. The space of
all sibling $\si_d$-invariant laminations is compact. All
q-laminations are
sibling $\si_d$-invariant.
\end{thm}

While, by Theorem~\ref{t:siblinv}, all sibling $\si_d$-invariant
laminations are invariant in the sense of Thurston, it
is easy to see that the opposite is
not true already for quadratic laminations (see an example in \cite{bmov13} right after Lemma 3.20). Unless specified otherwise,
in what follows by \emph{($\si_d$-)invariant} laminations we mean \emph{sibling ($\si_d$-)invariant}
laminations.

Let us now discuss gaps in the context of $\si_d$-invariant laminations.
The \emph{degree} of an orientation preserving covering
 map of a topological circle is the number of preimages of a generic point.
This concept is further developed below.

\begin{dfn}\label{d:degree}
If $\ell$ is a leaf of $\lam$ then the degree of $\ell$ is $1$ if $\ell$ is not critical and $2$ if $\ell$ is critical.
Let $G$ be a gap of an invariant lamination; $G$ is \emph{all-critical} if all its edges are critical.
If $G$ is all-critical, the degree $\deg(G)$ of $G$ is the number of vertices of $G$. If $G$ is not all-critical,
the degree $\deg(G)$ of $G$ is the number of components in the set $\si_d^{-1}(x)\cap \bd(G)$ for any point $x\in \si_d(G')$
($\si_d$ is considered only on $\bd(G)$).
\end{dfn}

Let us also define (pre)periodic gaps.

\begin{dfn}[Periodic and (pre)periodic gaps]\label{d:gaps-i}
Let $G$ be a gap of an invariant lamination $\lam$.
A gap/leaf $U$ of
$\lam_\sim$ is said to be \emph{{\rm(}pre{\rm)}periodic} of period $k$
if $\si_d^{m+k}(U')=\si_d^m(U')$ for some $m\ge 0$, $k>0$; if $m, k$
are chosen to be minimal, then if $m>0$, $U$ is called
\emph{preperiodic}, and, if $m=0$, then $U$ is called \emph{periodic
$($of period $k)$}. If the period of $G$ is $1$, then $G$ is said to be
\emph{invariant}. We define \emph{precritical} and
\emph{{\rm(}pre{\rm)}critical} objects similarly to how (pre)periodic
and preperiodic objects are defined above.
\end{dfn}

A more refined series of definitions deals with infinite periodic
gaps of $\si_d$-invariant laminations. There are three
types of such gaps: \emph{caterpillar} gaps, \emph{Siegel} gaps, and
\emph{Fatou gaps of degree greater than one}. We define them below.
Observe that, by \cite{kiw02}, infinite gaps eventually map onto
periodic infinite gaps.

\begin{lem}\cite[Lemma 2.28]{bopt20}\label{l:edges} Any edge of a (pre)periodic gap is either
(pre)periodic or (pre)critical.
\end{lem}

Let us now classify infinite gaps.

\begin{dfn}\label{d:fatou}
A periodic \emph{Fatou gap of period $n$ is of degree $k>1$} if the degree of
$\si_d^n|_{\bd(G)}$ is $k>1$. Any periodic Fatou gap of degree greater than $1$ is said to be
\emph{hyperbolic}. If the degree of a Fatou gap $G$ is $2$ (resp., $3$),
then $G$ is said to be \emph{quadratic} (resp., \emph{cubic}).
\end{dfn}

The next lemma is well-known (similar results are obtained in
\cite{blo86, blo87a, blo87b} for ``graph'' maps).

\begin{lem}\label{l:fatou}
Let $G$ be a hyperbolic gap of period $n$ and of degree $k>1$. Then the map
$\si_d^n|_{\bd(G)}$ is monotonically semiconjugate to $\si_k$ by a map $\psi:\bd(G)\to \uc$
such that $\psi(G\cap\uc)=\uc$.
\end{lem}

Surprisingly, gaps of degree 1 have a bit more complicated properties.

\begin{dfn}\label{d:caterpillar} An infinite  gap $G$ is said
to be a \emph{caterpillar} gap if its basis $G\cap\uc$ is countable.
%(see Figure~6).
\end{dfn}

As as an example, consider a periodic gap $Q$ such that:
\begin{itemize}
 \item  The boundary of $Q$ consists of a periodic leaf
     $\ell_0=\ol{xy}$ of period $k$, a critical leaf
     $\ell_{-1}=\ol{yz}$ concatenated to it, and a countable
     concatenation of leaves $\ell_{-n}$ accumulating at $x$ (the
     leaf $\ell_{-r-1}$ is concatenated to the leaf $\ell_{-r}$,
     for every $r=1$, 2, $\dots$).
\item We have $\si^k(x)=x$, $\si^k(\{y, z\})=\{y\}$, and $\si^k$
    maps each $\ell_{-r-1}$ to $\ell_{-r}$ (all leaves are shifted
    by one towards $\ell_0$ except for $\ell_0$, which maps to
    itself, and $\ell_{-1}$, which collapses to the point $y$).
\end{itemize}

The description of $\si_3$-invariant caterpillar gaps is in
\cite{bopt16b}. In general, the fact that the basis of a
caterpillar gap $G$ is countable implies that there are lots of
concatenated edges of $G$. Other properties of caterpillar gaps can
be found in Lemma~\ref{l:cater}.

\begin{lem}\label{l:cater}
Let $G$ be a caterpillar gap of period $k$. Then the degree of
$\si_d^k|_{\bd(G)}$ is one, $G\cap\uc$ contains at least one periodic point and, if there are multiple periodic points,
all have the same period.
\end{lem}

Recall that $\si_d$ has been extended over all edges of $G$.

\begin{proof}
We may assume that $k=1$. %Consider $\si_d|_{\bd(G)}$.
By Lemma \ref{l:fatou},
if the degree $r$ of $\si_d|_{\bd(G)}$ is greater than one, then $\psi(G\cap\uc)=\uc$
and, hence, $G\cap\uc$ is uncountable, a contradiction.

If the degree of $\si_d|_{\bd(G)}$ is one, then it is known
\cite{ak79, blo84} that either (1) $\si_d|_{\bd(G)}$ is
monotonically semiconjugate to an irrational rotation by a map
$\psi$, or (2) $\si_d|_{\bd(G)}$ has periodic points. Take the set
$B$ of all points of $\bd(G)$ that do not belong to open segments
in $\bd(G)$, on which $\psi$ is a constant. In case (1) $\psi$
collapses the edges of $G$ to points
because otherwise there would exist a finite union of
their $\psi$-images covering the whole $\uc$ while by
Lemma~\ref{l:edges} any edge of $G$ eventually maps to a point or to
a periodic edge of $G$. Hence $B$ is uncountable contradicting the
definition of a caterpillar gap. Thus, (2) holds.
\end{proof}

The next lemma is well-known.

\begin{lem}\label{l:siegel}
Let $G$ be a periodic Siegel gap of period $n$. Then $G\cap \uc$ contains
no periodic points while $G$ has an edge whose eventual image is critical.
\end{lem}

In this paper we study limits of q-laminations.

\begin{dfn}\label{d:rigileaf} A leaf/gap $G$ of a lamination $\lam$ is \emph{rigid}
if any q-lamination close to $\lam$ has $G$ as its leaf/gap.
\end{dfn}

We studied rigidity for laminations in \cite{bopt16}. For the sake of partial self-containment we reprove
a couple of claims from \cite{bopt16} here but refer to \cite{bopt16} for  other claims.

\begin{lem}[Lemmas 2.4-2.10 of \cite{bopt16}]\label{l:rigid}
Any periodic leaf of
$\lam$ is  rigid. Moreover, any hyperbolic periodic gap of $\lam$ that has no eventual image with a critical edge is rigid too.
\end{lem}

\begin{proof}
We may assume that $\lam$ is the limit lamination of a sequence of $\si_d$-invariant q-laminations. Then
the claim follows from Lemmas 2.4-2.10 of \cite{bopt16}. %However for the sake of completeness
Here we will provide a
proof only in a particular case. Let $G$ be a gap of $\lam$ such that
$\si_d^n(G)=G$ and $\ell=\ol{ab}$ is an edge of $G$ such that $\si_d^n(\ell)=\ell$. We claim that $\ell$ is rigid.
Choose a compact  subset  $H$ in the interior of $G$. If a q-lamination $\hlam$ is close to $\lam$ then there exists a unique gap $G(\hlam)\supset H$.
Moreover, $G(\hlam)$ has a unique edge $\hell$ that  is close to  $\ell$. Let us
describe the location of $\hell$ with respect to $\ell$. Evidently, as $\hlam\to \lam$, $\hell\to \ell$. Also, $\hell$ cannot cross $\ell$ as otherwise
$\si_d^n(\hell)$ will cross itself. If $\hell$ is located on the same side of $\ell$ as $H$ then $\si_d^n(\hell)$ will cross $G(\hlam)$, a contradiction.
Hence if $\hell\ne \ell$ then the only possibility is that $\hell$ is located outside $G$ and is repelled away from $\ell$ by $\si_d^n$. However in this case
we see that $\si_d^n(G(\hlam))\ne G(\hlam)$ while the interiors of $\si_d^n(G(\hlam))$ and $G(\hlam)$ are non-disjoint, a contradiction. It follows that
a periodic edge of a gap is rigid. Similar arguments are used in the proofs of Lemmas 2.4-2.10 \cite{bopt16} in which a few claims, including
the above quoted ones, are proven.
\end{proof}

\subsection{Cones}
The following definition was given in \cite{bopt16}.

\begin{dfn}\label{d:cone}
A family $C$ of leaves $\ol{ab}$ sharing the same endpoint $a$
is said to be a \emph{cone (of leaves of $\lam$)}. The point $a$ is
called the \emph{vertex} of the cone $C$ while all other points of $\uc\cap\bigcup C$ are called
the \emph{endpoints} of $C$. The entire set $\uc\cap \bigcup C$ is called
the \emph{basis} of the cone $C$ and is denoted by $C'$. We will often identify $C$ with $\bigcup C$.
A cone is said to be \emph{infinite} if it consists of infinitely many leaves.
\end{dfn}

We need results of \cite{bopt16} dealing with cones of invariant laminations.
Lemma \ref{l:inficon2} summarizes Lemmas 2.13, 2.14 and 2.15 from \cite{bopt16} and relies upon \cite{bmov13}.
However, first we state a basic lemma  which is instrumental in the proof of Lemma \ref{l:inficon2}.

\begin{lem}\cite[Corollary 3.7]{bmov13}\label{l:order} Let $\lam$ be a $\si_d$-invariant lamination.
If leaves $\ol{ab}, \ol{ac}\in \lam$ are such that $\si_d$ is injective on $\{a, b, c\}$ then $\si_d$ preserves the
circular order of $\{a, b, c\}$.
\end{lem}

In what follows the symbol $<$ indicates the counterclockwise (positive) circular order among points on the circle.

\begin{lem}\label{l:inficon2}
Let $\lam$ be a $\si_d$-invariant lamination. Let $C$ be a cone
of $\lam$ with periodic vertex $v$ of period $n$. Then all periodic endpoints of $C$
are of period $n$, and if $C$ is finite then all its endpoints
are periodic (of period $n$). Suppose that $C$ is infinite.
Let $\ol{va_1}, \dots, \ol{va_k}$ be all $n$-periodic leaves in $C$ labeled so that
$v=a_0<a_1<\dots<a_k<v=a_{k+1}$ and $\si_d^n(a_i)=a_i$ for each $i$.
If, for some $i$, $C'\cap (a_i, a_{i+1})\ne \0$, then one of the
following holds.

\begin{enumerate}

\item The map $\si_d^n$ moves all points of $C'\cap (a_i, a_{i+1})$
    in the positive direction within $(a_i, a_{i+1})$ until after finitely many iterations a point
    is mapped either to $a_{i+1}$ or to $v$.

\item The map $\si_d^n$ moves all points of $C'\cap (a_i, a_{i+1})$
    in the negative direction within $(a_i, a_{i+1})$ until a point
    is mapped either to $a_i$ or to $v$.

\item There exist two points $u, w\in C'$ with $a_i<u\le w<a_{i+1}$ such
that $\si_d^n(u)=\si_d^n(w)=v$, $C'\cap (u, w)=\0$, the map
$\si_d^n$ sends all points of $(a_i, u]$ in the positive direction within $(a_i, a_{i+1})$
except for those, which are mapped to $v$, and all points of $[w,
a_{i+1})$ in the negative direction within $(a_i, a_{i+1})$ except for those, which are
mapped to $v$ so that all points of $(a_i, a_{i+1})\cap C$ are eventually mapped to $v$.

\end{enumerate}

\end{lem}

\section{Convergence of Side laminations}\label{s:main}

We start by discussing some properties of invariant laminations containing leaves $\ol{0\frac 13}$, $\ol{0\frac 23}$.

\begin{lem}\label{l:only} Suppose that $\lam'$ is a cubic invariant lamination with
 $\ol{0\frac 13},\ol{0\frac 23}\in\lam'$
and $\ell\in\lam'$ is a leaf. Then the following holds.

\begin{enumerate}

\item If $\ell$ is disjoint from $0\in\uc$ then the three pullbacks of $\ell$ in $\lam'$
are unique. % (hence there is a unique full sibling collection of leaves in %$\lam'$ whose image is $\ell$).

\item Any leaf $\ell\in\lam'$ eventually maps either to $0$ or to the leaf $0\frac12$.

\end{enumerate}

\end{lem}

\begin{proof} (1) Left to the reader.

(2) Assume that $\ol{0\frac12}\in \lam'$. Consider the cone $C$ of leaves of $\lam'$ with vertex $0$.
Then by Lemma \ref{l:inficon2} the leaves $\ol{0 x}\in C$ are either eventually mapped to a critical leaf with an endpoint
$0$, or to $\ol{0\frac12}\in \lam'$. Evidently, this applies to all leaves that eventually map to a leaf of $C$. Moreover,
the only other critical leaf that can belong to $\lam'$ is $\ol{\frac13 \frac23}$. Hence any precritical leaf of $\lam'$
eventually maps to $0$. Thus, if a leaf is precritical or eventually maps to a leaf with and endpoint $0$,
then the claim of the lemma holds.

Suppose now that $\ell\in \lam'$ is a leaf that is not precritical and never maps to a leaf with an endpoint $0$.
Since eventual images of $\ell$ cannot cross $\ol{0\frac13}$ or $\ol{0\frac23}$
then the endpoints of all eventual images of $\ell$ belong either to $(0, \frac13)$, or to $(\frac13, \frac23)$, or to $(\frac23, 0)$.
However this is impossible because under our assumptions both endpoints of $\ell$ have the same itinerary with respect to the arcs
$(0, \frac13)$, $(\frac13, \frac23)$, $(\frac23, 0)$ which implies that they have to coincide.

A similar argument shows that the same holds if $\ol{0\frac12}\notin \lam'$ (with the obvious correction that
then all leaves eventually map to $0$).
\end{proof}

From now on, we consider the following Basic Setup.

\medskip

\noindent\textbf{Basic Setup.} \emph{Consider a sequence of Side q-laminations $\lam_n$
converging to an invariant lamination $\lam$.
Let $C_n$, $D_n$ be the critical sets of $\lam_n$, and suppose that
some critical chords $\oc_n\subset C_n$, $\od_n\subset D_n$ converge to $\oc_\circ=\ol{0\frac 13}$, $\od_\circ=\ol{0\frac 23}$, respectively.
Also, assume that $\Pc_n$ are critical portraits compatible with $\lam_n$ and converging to
a critical portrait $\Pc$.
Denote the triangle $\ch(\oc_\circ\cup\od_\circ)$ by $\rhd$.
Assume that $C_n\to C$ and $D_n\to D$.}

\medskip

Recall that a lamination is \emph{unicritical} if it has a unique critical set
(thus, in the cubic case, the critical set is invariant %with respect to
under rotation by $\frac13$).
We assume that $C_n=D_n$ only if $\lam$ is unicritical.

\begin{lem}
\label{l:lam-uni}
If $\hlam$ is unicritical, nonempty and compatible with $\oc_\circ$, $\od_\circ$, then $\rhd$ is a gap of $\hlam$.
\end{lem}

\begin{proof}
Let $C$ be the unique critical set of $\hlam$; then $\rhd\subset C$.
We claim that $\rhd=C$.
Indeed, by way of contradiction, suppose that $\rhd\subsetneqq C$.
Then no edge of $\rhd$ is a leaf of $\hlam$ as otherwise the fact that
$\hlam$ is unicritical would imply that $\rhd=C$.
Consider two cases.

(1) There exist leaves $\ell=\ol{0 x}\in \hlam$. Without loss of generality we may assume $0<x<\frac13$. Applying $\si_3$ a few times if necessary
we may assume that $\frac19\le x<\frac13$.
Since no edge of $\rhd$ is a leaf of $\hlam$ and no leaf of $\hlam$ can cross $\ol{\frac13 \frac23}$
then we may assume that $\frac29<x<\frac13$.
Hence $\frac23<3x<\frac23+x$.
However, this implies that the leaf $\si_3(\ell)$ crosses the leaf $\ol{\frac23 (\frac23+x)}$ (the latter is a leaf of
$\hlam$ because $\hlam$ is unicritical), a contradiction.

(2) There are no leaves with endpoint $0$ in $\hlam$.
Note that $C$ is then an infinite invariant gap of degree 3, which implies that $C=\ol\disk$,
a contradiction with $\hlam$ being nonempty.
\end{proof}

\begin{lem}\label{l:uni3}
Assume the Basic Setup and that all $\lam_n$ are \emph{unicritical}.
Then $\rhd$ is a gap of $\lam$, and $\Pc$ consists of two edges of $\rhd$. %$\ch(\oc_\circ\cup\od_\circ)$.
\end{lem}

\begin{proof}
The claim follows immediately from Lemma \ref{l:lam-uni}.
\end{proof}

Since $\lam_n$ are Side laminations, $0$ does not belong to their critical sets' boundaries.
By Lemma \ref{l:uni3} we may assume that all $\lam_n$ are not unicritical.

Note that $\oc_n$ and $\od_n$, with $n$ large, cannot simultaneously cross the horizontal
 diameter as then they would cross each other.

\begin{dfn}[Passive]
Passing to a subsequence, we may assume that either all $\oc_n$ are above the horizontal diameter,
or all $\od_n$ are below the horizontal diameter.
In the former case, say that $\oc_n$ are \emph{passive},
and, in the latter case, that $\od_n$ are \emph{passive}.
\end{dfn}

Recall that $T^0$ is the invariant caterpillar gap with edge $\oc_\circ$ lying above $\oc_\circ$.

\begin{lem}
\label{l:T0}
Assume the Basic Setup.
If all $\oc_n$ are passive, then $T^0$ is a gap of $\lam$,
 in particular, $\oc_\circ\in\lam$.
\end{lem}

\begin{proof}
Consider $\oc_n=\ol{a_nb_n}$ for large $n$; we may assume that $0<a_n<b_n<\frac 12$.
Recall that there is a piecewise linear extension $s_n$ of $\si_3$ to $\ol\disk$ that takes leaves and gaps of $\lam_n$ to leaves and gaps.
The map $s_n$ takes the part $A$ of $\ol\disk$ above $\oc_n$ onto the entire disk and is
in fact a weakly polynomial-like map of degree one in the sense of \cite{GM93} with isolated fixed points
of positive Lefschetz indices.
By Lemma 3.7 of \cite{GM93}, there is a unique fixed point of $s_n$ in the interior of $A$;
this fixed point must necessarily lie in a finite invariant rotational gap or leaf $G_n$ of $\lam_n$
(here, we use that $\oc_n$ are passive, therefore, there are no fixed points
of $s_n|_A$ on the unit circle). Alternatively, the existence of a gap $G_n$ follows from well-known properties
of q-laminations.
Clearly, the limit of $G_n$ is an invariant gap of $\lam$ of degree one with a vertex at $0$.
This can only be a caterpillar gap that is non-strictly above $\oc_\circ$, and the only possibility is $T^0$.
\end{proof}

Consider the case when the horizontal diameter is a leaf of $\lam$.

\begin{lem}\label{l:12}
Assuming the Basic Setup, suppose also that $\ol{0 \frac12}\in \lam$.
Then $C=\oc_\circ$ and $D=\od_\circ$.
\end{lem}

\begin{proof}
By Lemma \ref{l:rigid}, we may assume that $\ol{0\frac12}\in\lam_n$ for all $n$.
Hence all $\oc_n$ are passive and all $\od_n$ are passive.
By Lemma \ref{l:T0} applied to $\oc_n$ we have that $T^0$ is a gap of $\lam$ and
$\oc_\circ\in\lam$. Similarly, $\tau(T^0)$ is a gap of $\lam$ and
$\od_\circ\in\lam$. This implies that $C=\oc_\circ$ and $D=\od_\circ$ as desired.
\end{proof}

Theorem \ref{t:main-lam} follows from Propositions \ref{p:limlam} and \ref{p:limlam1}.

\begin{prop}
  \label{p:limlam}
Assume the Basic Setup. Assume also that all $\oc_n$ are passive
while $D$ is a gap rather than a leaf. Then
$\lam$ coincides with $\lam^k$ or $\lam^k_\square$, for some positive integer $k$.
\end{prop}

\begin{proof}
By Lemma \ref{l:T0}, the set $T^0$ is an invariant gap of $\lam$.
If $T^i$ are gaps of $\lam$, for all $i\ge 0$, then
$\ol{0\frac 12}\in\lam$, and, by Lemma \ref{l:12}, we must have $\si_3(D)=\{0\}$, contrary to our assumption.
Hence we can choose the largest positive integer $k<\infty$ with the property that $T^0$, $\dots$, $T^{k-1}$ are gaps of $\lam$ while $T^{k}$ is not a gap
of $\lam$.

Either there is a two-to-one pullback $H^{k}$ of $T^{k-1}$ in $\lam$, or
there are two one-to-one pullbacks of $T^{k-1}$ in $\lam$.
In the former case, $D$ must coincide with $H^{k}$, hence $\lam=\lam^k$
(recall that $\lam^k$ is the unique invariant lamination with critical sets $\oc_\circ:=\ol{0\frac 13}$ and $H^k$).
In the latter case the two pullbacks of $T^{k-1}$ in $\lam$ are gaps of $\lam$ contained in $H^{k}$.
The set $H^{k}$ consists of these two gaps and a critical quadrilateral $Q$ between them.
We claim that $Q$ is the set $Q^k$ defined in Section \ref{ss:ans}, that is,
 the top edge of $Q$ is the bottom edge of $T^{k-1}$.
Firstly, our assumptions imply that $\ol{0\frac 23}$ is compatible with $Q$;
 this leaves exactly two options for $Q$, namely, critical quadrilaterals in $H^k$ for which $\ol{0\frac 23}$ is a diagonal.
Secondly, $T^k$ cannot be compatible with $Q$ as otherwise $T^k$ would be a gap of $\lam$.
This second consideration rules out one of the two options. Hence, $Q=Q^k$, as claimed.
Now, three of the four edges of $Q^{k}$ are leaves of $\lam$.
By the sibling property, the fourth edge of $Q^{k}$ must also be a leaf of $\lam$.

Suppose that $Q^k$ is a gap of $\lam$.
There is only one invariant lamination with critical sets $Q^k$ and $\oc_\circ$; this is $\lam^k_\square$;
the uniqueness of such lamination follows from the fact that all eventual pullbacks of $\oc_\circ$
in complementary components of $T^0\cup\dots\cup T^{k-1}\cup Q^k$ are uniquely defined.
Hence, $\lam=\lam^k_\square$ in this case.
Finally, it may happen that $Q^k$ is not a gap of $\lam$ but all edges of $Q^k$ are leaves of $\lam$.
In this case, $Q^k$ is the union of two gaps of $\lam$ separated by a diagonal;
the only option for $D$ in this case is $\ol{0\frac 23}$ but this option is explicitly excluded by our assumptions.
\end{proof}

As the following proposition claims, all options for $\lam$ listed in Proposition \ref{p:limlam} realize.

\begin{prop}
  \label{p:limlam1}
  Let $k$ be a positive integer.
Both $\lam^k$ and $\lam^k_\square$ can be obtained as limits of sequences of dendritic
 post-critically finite laminations.
\end{prop}

Note that dendritic post-critically finite laminations are laminations
 of dendritic post-critically finite polynomials, as follows from \cite{BFH92}.

\begin{proof}
Let us introduce a series of sequences of dendritic postcritically finite laminations $\{\lam_N^k\}_{N=1}^\infty, k=1, 2, \dots$.
Firstly, for every integer $N>2$, let $G_N$ be the unique finite invariant gap with $N$ vertices located above $\ol{0\frac12}$
on which $\si_3$ acts with combinatorial rotation number $\frac1{N}$. Denote the vertex of $G_N$ closest to $0$ by $x_N$. Then
the vertex $u_N=3^{N-1}x_N\in [\frac13, \frac12]$ of $G_N$ maps to $x_N$ in one step and equals $\frac13+\frac{x_N}{3}$.
Hence $x_N=\frac{1}{3^N-1}$ and $u_N=\frac13+\frac{1}{3(3^N-1)}=\frac{3^{N-1}}{3^N-1}$.
We will also write $T^0_N$ instead of $G_N$.

Define a gap $T^1_N$ that maps onto $T^0_N=G_N$ one-to-one and has a vertex $\frac{x_N}{3}=z^1_N$ and all other vertices belonging to $[\frac13, \frac12]$.
One can visualize $T^1_N$ by thinking of pulling $G_N$ back down towards the repelling diameter $\ol{0\frac12}$. In fact the process of pulling $G_N$ back down towards
$\ol{0\frac12}$ can be repeated giving rise to gaps $\{T^k_N\}_{k=1}^\infty$ that become more and more flat and horizontal as $k\to \infty$.
By construction, $\si_3(T^k_N)=T^{k-1}_N$ so that $\si_3^k(T^k_N)=G_N$.
Each of these gaps has $N$ vertices.
See Fig. \ref{fig:T4k} for an illustration.
Observe that the unique vertex of
$T^k_N$ in $(0, x_N)$ is $z^k_N=\frac{x_N}{3^k}=\frac{1}{3^k(3^N-1)}.$ Set $\od^k_N=\ol{z^k_N\, z^k_N+\frac23}$.

\begin{figure}
  \centering
  \includegraphics[width=.7\textwidth]{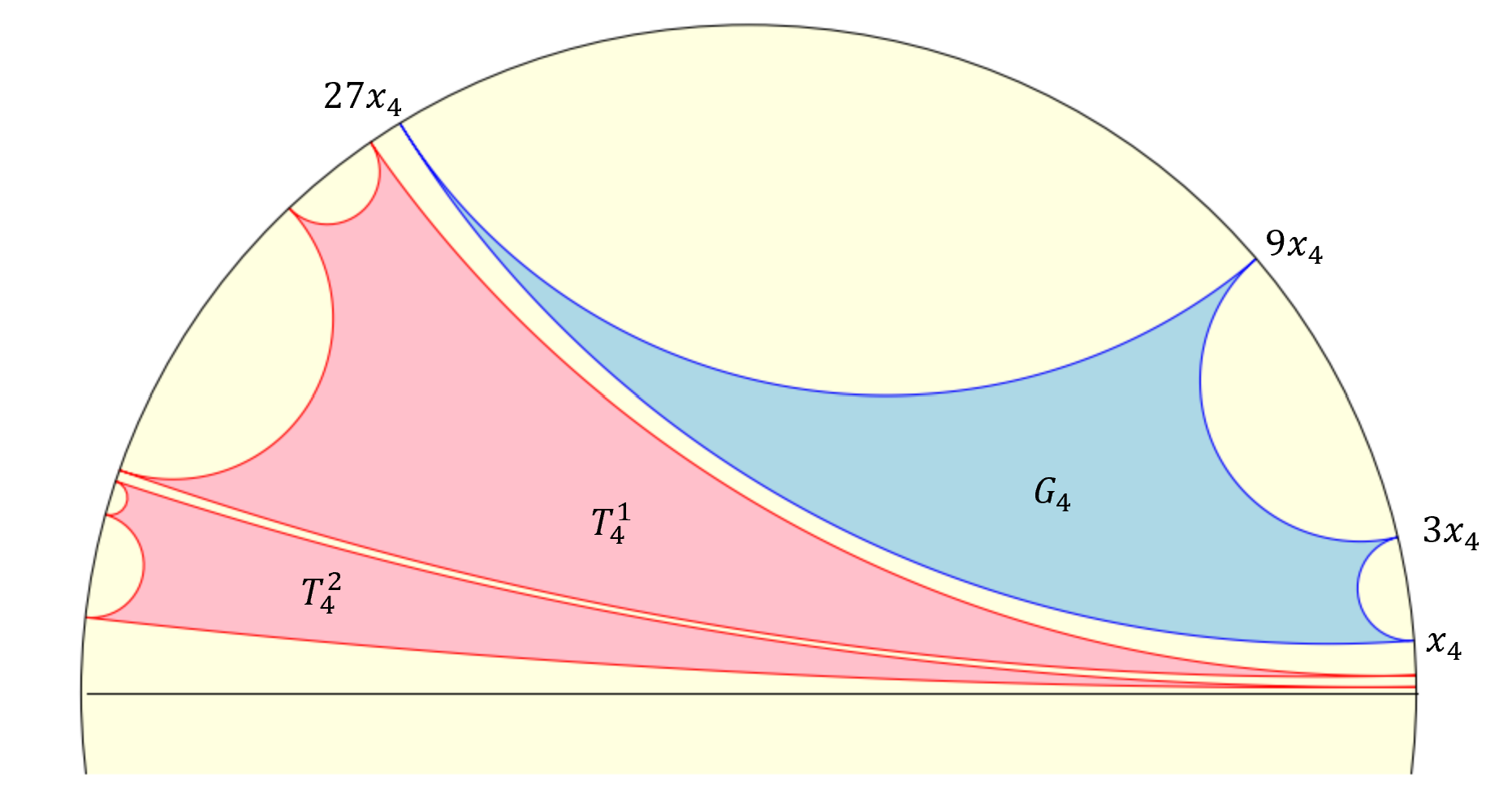}
  \caption{The invariant gap $G_4$ and its pullbacks $T^k_4$ (only $k=1$ and $k=2$ are shown).}\label{fig:T4k}
\end{figure}

Define $H^{k}_N$ as the convex hull of the full $\si_3|_{[\frac 13,1]}$-preimage of $T^{k-1}_N\cap\uc$.
Also, take $\oc_N$ to be a strictly preperiodic critical chord separating $T^0_N$ from $T^1_N$.
For the sake of definiteness we may choose $\oc_N$
so that it does not depend on $k$ and is eventually mapped to a fixed point.
Namely, choose a point $y_N=\frac{1}{2\cdot 3^N}\in (0, x_N)$ so that $\si_3^N(y_N)=\frac12$,
 and set $\oc_N=\ol{y_N\,(y_N+\frac13)}$.
The thus constructed critical portrait $\oc_N$, $\od_N^k$ is shown in Fig. \ref{fig:H41} for
 the case $N=4$ and $k=1$.
\begin{figure}
  \centering
  \includegraphics[width=.6\textwidth]{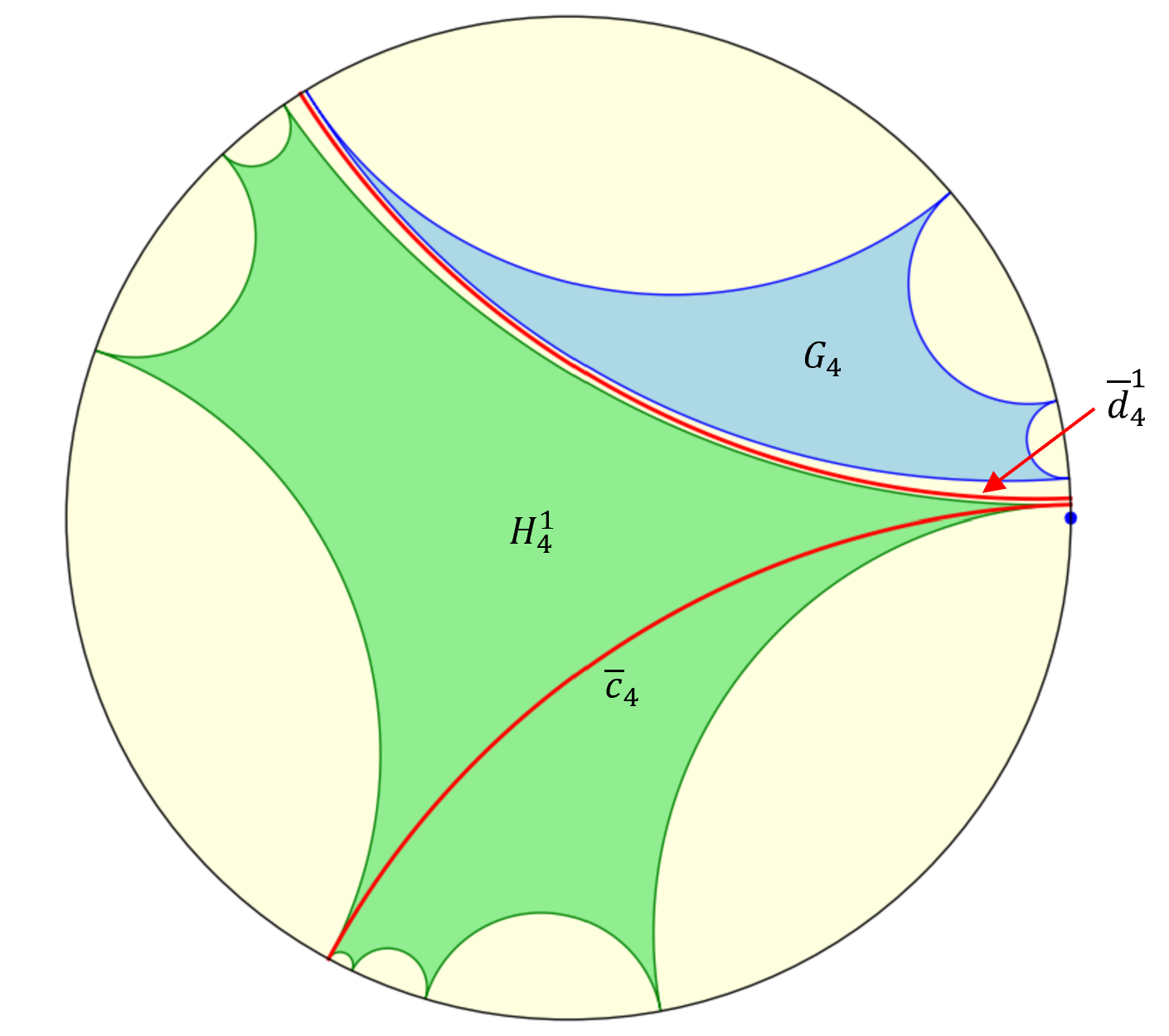}
  \caption{The gaps $G_4$ and $H^1_4$ of the lamination $\lam_4^1$.
  Also shown is a critical portrait $\{\oc_N,\od_4^1\}$ compatible with this lamination.}
  \label{fig:H41}
\end{figure}

The two finite critical sets $H^k_N$ and $\oc_N$ we have constructed are pre-periodic.
It is well-known that in such cases there exists a unique q-lamination $\lam^k_N$ with such critical sets.
To show that one can use Thurston's pullback construction \cite{thu85}
(observe that the critical sets in question have well defined pullbacks). Alternatively,
it follows from Kiwi's results \cite{kiw04}. Notice that the most general situation is considered in a recent preprint
\cite{bot26}.
Note that the lamination $\lam^k_N$ is automatically dendritic.
Clearly, the limit of $H^k_N$ coincides with $H^k$ while the limit of $\oc_N$ is $\oc_\circ:=\ol{0\frac 13}$.
It follows that %the
$\lim_{N\to\infty} \lam_N=\lam^k$, as $\lam^k$ is the unique
invariant lamination with critical sets $\oc_\circ$ and $H^k$.

Suppose that $k>1$, and set $\ell_N$ to be the one-to-one $\si_3^k$-pullback of $\oc_N$ that separates $T^{k-1}_N$ from $T^{k-2}_N$.
If $k=1$, let $\ell_N$ be an eventual pullback of $\oc_N$ ``under'' the shortest edge of $T^0_N$,
 i.e. $\ell_N$ is separated by the shortest edge of $T^0_N$ from the rest of $T^0_N$.
Define $Q^k_N$ to be the two-to-one pullback of $\ell_N$ unlinked with $\oc_N$.
Then there is a unique dendritic lamination $\lam_{N,\square}$ with critical sets $\oc_N$ and $Q^k_N$,
 and the limit of $\lam_{N,\square}$ coincides with $\lam^k_\square$ since the latter is
 the unique invariant lamination with critical sets $\oc_\circ=\lim\oc_N$ and $Q^k=\lim Q^k_N$.
\end{proof}

Theorem \ref{t:main1} and Corollary \ref{c:main1}
 follow from Proposition \ref{p:nec}. %and \ref{p:suf}.

\begin{prop}
\label{p:nec}
Assume the Basic Setup. Assume also that $D$ is a gap and not a leaf. Then $D$ must coincide with $H^k$ or $Q^{k}$,
for some positive integer $k$. All these options realize.
\end{prop}

\begin{proof}
By Lemma \ref{l:T0} applied to $\od_n$, the critical chords $\od_n$ are not passive, hence $\oc_n$ are.
Proposition \ref{p:limlam} is then applicable.
Since $D$ is the critical set of $\lam$ distinct from $\ol{0\frac 13}$,
 the only options for $D$ are $H^k$ or $Q^k$, for some $k$
 (the first option realizes when $\lam=\lam^k$ while the second option corresponds to $\lam=\lam^k_\square$).
The last claim of the proposition follows from Proposition \ref{p:limlam1}.
\end{proof}

Recall that $\Delta'$ is the full $\si_3|_{[\frac 13,1]}$-preimage of $\ol{(T^0\cup T^1\cup\dots)}\cap\uc$,
and $\Delta=\ch(\Delta')$.
We have shown that in any USC partition of the space of critical portraits, respecting the partition of the space of Side critical portraits,
all critical portraits in $\Delta$ or in $\tau(\Delta)$ must be contained in one class.  Let us now address the question as
to which q-laminations these critical portraits correspond to. All of them have one critical chord with a fixed endpoint, called  $\oc$,
and a second critical chord $\od$. We may assume that $\oc=\ol{0\frac13}$. Then the second critical chord $\od$ has a finite forward orbit
contained in the closure of the component of
$\disk\setminus \od$ that contains $0$.  Assume first that $\si_3(\od)\not\in \{0,\frac12\}$. Then  $\od$ is the major (the longest edge) of a quadratic
invariant gap $Q$ that contains $\oc$;  such gaps were studied  in \cite{bopt16b}. It was shown in this paper that, given $\oc$ and $\od$ as above,
there exists a unique q-lamination that contains $Q$ as an invariant quadratic gap.

Consider the remaining cases. First assume that $\si_3(\od)=0$. This can happen in two cases. First, $\od$ may be equal to $\ol{\frac13 \frac23}$. In this case
by \cite{bopt16b} we have a quadratic invariant gap $Q$ that contains $\oc$ and a unique q-lamination with $Q$ as its invariant gap.
On the other hand, if $\od=\ol{0 \frac23}$, then the corresponding q-lamination is the empty lamination. Finally, if $\si_3(\od)=\frac12$,
then the corresponding q-lamination contains two quadratic invariant
gaps that share the fixed leaf $\ol{0\frac12}$ on their boundaries. One of the gaps is located above $\ol{0\frac12}$ and the other one is
located below $\ol{0\frac12}$.

\bibliographystyle{amsalpha}

\end{document}